\documentclass[onefignum,onetabnum]{siamart171218}

\usepackage{float}
\usepackage{amsfonts}
\usepackage{amssymb}
\usepackage{graphicx}
\usepackage{epstopdf}
\usepackage{algorithmic}
\ifpdf
  \DeclareGraphicsExtensions{.eps,.pdf,.png,.jpg}
\else
  \DeclareGraphicsExtensions{.eps}
\fi

\usepackage{bm}
\usepackage{tikz-cd}
\usepackage{subfig}
\usepackage{mathtools}
\usepackage{stmaryrd}
\usepackage{multirow}
\usepackage{enumitem}
\usepackage{esint}

\newsiamremark{remark}{Remark}
\newsiamremark{assumption}{Assumption}
\newsiamthm{property}{Property}

\newsiamremark{hypothesis}{Hypothesis}
\crefname{hypothesis}{Hypothesis}{Hypotheses}
\newsiamthm{claim}{Claim}

\def\XXint#1#2#3{{\setbox0=\hbox{$#1{#2#3}{\int}$ }
\vcenter{\hbox{$#2#3$ }}\kern-.6\wd0}}

 \headers{Aux Space Preconditioners for PDEs with Cordes
 Coefficients}{G. Gao, S. Wu}

\title{Auxiliary Space Preconditioners for $C^{0}$ Finite Element
Approximation of Hamilton--Jacobi--Bellman Equations with Cordes
Coefficients
 }

\author{
Guangwei Gao \quad \quad 
Shuonan Wu\thanks{Corresponding author. School of Mathematical Sciences,
 Peking University, Beijing 100871, China 
.}
}

\usepackage{amsopn}

\begin{document}

\maketitle
\begin{abstract}
  In the past decade, there are many works on the finite element
  methods for the fully nonlinear Hamilton--Jacobi--Bellman (HJB)
  equations with Cordes condition.  The linearised systems have large
  condition numbers, which depend not only on the mesh size, but also
  on the parameters in the Cordes condition. This paper is concerned
  with the design and analysis of auxiliary space preconditioners for
  the linearised systems of $C^0$ finite element discretization of HJB
  equations [Calcolo, 58, 2021].  Based on the stable
  decomposition on the auxiliary spaces, we propose both the additive
  and multiplicative preconditoners which converge uniformly in the sense
  that the resulting condition number is independent of both the
  number of degrees of freedom and the parameter $\lambda$ in Cordes
  condition. Numerical experiments are carried out to illustrate the
  efficiency of the proposed preconditioners. 
\end{abstract}



\section{Introduction}
Let $\Omega$ be a bounded, open, convex polytopal domain in
$\mathbb{R}^{d}$, where $d = 2,3$ represent the dimension. In this
paper, we are interested in the Hamilton--Jacobi--Bellman (HJB)
equations of the following type:
\begin{equation} \label{eq:HJB} 
\sup_{\alpha \in \Lambda} (L^\alpha u - f^\alpha) = 0 
\quad \text{in }\Omega, \qquad 
u = 0 \quad \text{on }\partial\Omega,
\end{equation}
where $\Lambda$ is a compact metric space, and  
$$ 
L^\alpha v := A^\alpha : D^2 v + \bm{b}^\alpha \cdot \nabla
v - c^\alpha v.
$$ 
Here, $D^2u$ and $\nabla u$ denote the Hessian and gradient of
real-valued function $u$, respectively.  The coefficient $ A^{\alpha}
\in C(\overline{\Omega} \times \Lambda; \mathbb{R}^{d\times d}) $ is
assumed to be uniformly elliptic, i.e., there exist constants $
\overline{\nu}, \underline{\nu} > 0 $ such that 
\begin{equation}\label{eq:HJB-elliptic}  
\underline{\nu} |\boldsymbol{\xi}|^{2} \leq
  \boldsymbol{\xi}^{t}A^{\alpha}(x)\boldsymbol{\xi} \leq
  \overline{\nu} |\boldsymbol{\xi}|^{2} \qquad \forall\boldsymbol{\xi}
  \in \mathbb{R}^{d}, \text{ a.e. in } \Omega, ~\forall \alpha \in
  \Lambda.
\end{equation} 
Further, $\boldsymbol{b}^{\alpha} \in C(\overline{\Omega} \times
\Lambda; \mathbb{R}^{d}) $ and $c^\alpha \geq 0, f^\alpha \in
C(\overline{\Omega} \times \Lambda; \mathbb{R})$.

The HJB equations arise in many applications including stochastic
optimal control, game theory, and mathematical finance
\cite{fleming2006controlled}. In \cite{maugeri2000elliptic,
smears2014discontinuous}, the HJB equations are shown to admit $H^2$
strong solutions under the following
Cordes condition.   
\begin{definition}[Cordes condition for \eqref{eq:HJB}]\label{lm:corde-HJB} 
The coefficients satisfy that there exist $ \lambda > 0 $ and $
  \varepsilon \in (0,1] $ such that 
\begin{equation}\label{eq:cordes-HJB} 
\displaystyle\frac{|A^{\alpha}|^{2} +
  |\boldsymbol{b}^{\alpha}|^{2}/2\lambda +
  (c^{\alpha}/\lambda)^{2}}{(\operatorname{tr}A^{\alpha} +
  c^{\alpha}/\lambda )^{2}} \leq \displaystyle\frac{1}{d + \varepsilon}
  \qquad \text{a.e. in }\Omega, ~\forall \alpha \in \Lambda.
\end{equation} 
\end{definition} 

In the past decade, several studies have been taken on the finite
element approximation of $H^2$ strong solutions of the HJB equations
with Cordes coefficients \eqref{eq:cordes-HJB}.  
The first discontinuous Galerkin (DG) method was proposed in
\cite{smears2014discontinuous}, which has been extended to the
parabolic HJB equations in \cite{smears2016discontinuous}.  The
$C^0$-interior penalty DG methods were developed in
\cite{neilan2019discrete}.  A mixed
method based on the stable finite element Stokes spaces was proposed
in \cite{gallistl2019mixed}.  Recently, the $C^0$ (non-Lagrange)
finite element method with no stabilization parameter was proposed in
\cite{wu2021c0}, where the element is required to be $C^1$-continuous
at $(d-2)$-dimensional subsimplex, e.g., $\mathcal{P}_k$-Hermite
family $(k\geq 3)$ in 2D and $\mathcal{P}_k$-Argyris family $(k \geq 5)$
\cite{neilan2015discrete, christiansen2018nodal} in 3D. 
The above discretizations can be naturally applied to the linear
elliptic equations in non-divergence form
\cite{smears2013discontinuous, kawecki2019dgfem, neilan2019discrete,
gallistl2019mixed, wu2021c0}. Other related topics include the unified
analysis of DGFEM and $C^0$-IPDG \cite{kawecki2020unified}, and the
adaptivity of $C^0$-IPDG \cite{brenner2020adaptive,
kawecki2020convergence}.

For all these discretizations, the discrete well-posedness is analysed
under the broken $H^2$-norm with possible jump terms across the
boundary.  This, after linearization, leads to the ill-conditioned
systems with condition number $\mathcal{O}(h^{-4})$ on quasi-uniform
meshes, where $h$ represents the mesh size.  Due to the similar
performance to the discrete system for fourth-order problems, it is
conceivable that the linearised system from HJB equations can be
effectively solved by the solvers for fourth-order problems, e.g.,
geometric multigrid \cite{peisker1987conjugate, brenner1989optimal,
brenner1999convergence, stevenson2003analysis, carstensen2021hierarchical} or domain decomposition
\cite{zhang1994multilevel, brenner1996two}.  In
\cite{smears2018nonoverlapping}, the nonoverlapping domain
decomposition preconditioner was studied for the DGFEM discretization
of HJB equations.

Traditional geometric multigrid methods depend crucially on the
multilevel structures of underlying grids. On unstructured grids, the
more user-friendly option is the algebraic multigrid method (AMG) that
have been extensively studied for the second-order equations.  In
\cite{peisker1988numerical}, the first biharmonic equation was
converted to a Poisson system based on the boundary operator
proposed in \cite{glowinski1979numerical}.  Under the framework of
auxiliary space preconditioning \cite{xu1996auxiliary}, Zhang and Xu
\cite{zhang2014optimal} proposed a class of optimal solvers based on
the auxiliary discretization of mixed form for the fourth-order
problems.  As a generalization of \cite{peisker1987conjugate,
peisker1988numerical, peisker1990iterative}, it works for a variety of
conforming and nonconforming finite element discretizations on both
convex and nonconvex domains with unstructured triangulation.  

The propose of this work is to study the auxiliary space
preconditioner to the $C^0$ finite element discretization of HJB
equations. More specifically, the numerical scheme for fully nonlinear HJB
equations leads to a discrete nonlinear problem that can be solved
iteratively by a semi-smooth Newton method
\cite{smears2014discontinuous, neilan2019discrete, wu2021c0}. The linear system obtained
from the semi-smooth Newton linearization are generally non-symmetric
but coercive. To handle the non-symmetry, the existing GMRES theory
\cite{elman1982iterative} will lead to a guaranteed minimum
convergence rate with a symmetric FOV-equivalent preconditioner
$P_{\lambda, h}$ that satisfies \eqref{eq:norm-equivalence}. The
construction of $P_{\lambda,h}$ under the auxiliary preconditioning
framework follows two steps:
\begin{enumerate}
\item Construct appropriate auxiliary spaces and corresponding
  transfer operators mapping functions from original space to the
    auxiliary spaces;
\item Devise solvers on auxiliary spaces so that the bounds in
  \eqref{eq:norm-equivalence} are uniform with respect to both $h$ and
    the parameter $\lambda$ in the Cordes condition. 
\end{enumerate}
Based on the stable decomposition for auxiliary spaces, both additive
and multiplicative preconditoners are shown to be efficient and
$\lambda$-uniform for the linearised system. Further, the
precondtioners only involve the Poisson-like solver which can be
efficiently solved by AMG with nearly optimal complexity.


In general cases, the auxiliary space preoconditioner is additive
\cite{xu1996auxiliary, zhang2014optimal, grasedyck2016nearly}, which
usually leads to a stable but relatively large condition number in
practical applications.  The first contribution of this work is the
construction and analysis of a multiplicative preconditioner based on
the specific structure of auxiliary spaces.  Having a coarse subspace,
the symmetrized two-level multiplicative precondition was shown to be
positive definite provided that the smoother on the fine level has
contraction property \cite{holst1997schwarz}. The condition number
estimate of multiplicative precondition at the matrix level can be
found in \cite{notay2013further}.  In this work, we show that the
contracted smoother together with the stable decomposition for
auxiliary spaces leads to a robust multiplicative precondititoner,
which is also numerical verified with better performance than the
additive version. 

The parameter $\lambda$ in the Cordes condition balances the diffusion
and the constant term. We emphasis that this parameter is not involved
in the monotonicity constant \eqref{eq:HJB-monotonicity}, which makes
it possible to consider the preconditioner with uniformity on
$\lambda$.  In this work, we carefully define the norm on the
auxiliary space so that the induced preconditioner is
uniform with respect to $\lambda$. Although the preconditioner is
designed for the $C^0$ finite element approximation, a similar idea
can be applied to other discretizations. 

The rest of the paper is organized as follows. In Section
\ref{sec:pre}, we establish the notation and state some preliminaries
results.  In Section \ref{sec:fov}, we apply the FOV-equivalence
preconditioner for the linear system, which can be used to solve
non-symmetric systems appearing in applications to the HJB
equations. In Section \ref{sec:solver}, we construct both the additive
and multiplicative auxiliary space preconditioners. We also show that
the condition numbers of the preconditioned systems are uniformly
bounded with the stable decomposition assumption, which is verified in
Section \ref{sec:analysis}. Several numerical experiments are
presented in Section \ref{sec:numerical} to illustrate the theoretical
results.

For convenience, we use $C$ to denote a generic positive constant
which may depend on $ \Omega $, share regularity of mesh and
polynomial degree, but is independent of the mesh size $ h $. The
notation $ X \lesssim Y $ means $ X \leq CY $. $ X \simeq Y $ means
$ X\lesssim Y $ and $ Y \lesssim X $. 

\section{Preliminaries} \label{sec:pre}
In this section, we first review the $H^2$ strong solutions to the HJB
equations \eqref{eq:HJB} under the Cordes condition
\eqref{eq:cordes-HJB}. Then we give a brief statement about the $C^0$
finite element scheme in \cite{wu2021c0}. 

Given an integer $ k\geq 0 $, let $ H^{k}(\Omega) $ and $
H^{k}_{0}(\Omega) $ be the usual Sobolev spaces, $
\|\cdot\|_{H^{k}(\Omega)} $ and $ |\cdot|_{H^{k}(\Omega)} $ denote the
Sobolev norm and semi-norm. We also denote $ V = H^{2}(\Omega)\cap
H^{1}_{0}(\Omega) $. For any Hilbert space $ X $, we denote $
X^{\prime} $ for the dual space of $ X $, and $\langle \cdot,
\cdot \rangle$ for the corresponding dual pair. We also denote
$|\cdot|$ as the Euclidian norm for vectors and the Frobenius norm for
matrices.

\subsection{\texorpdfstring{$H^2$}{H2} strong solutions to the HJB equations} We now invoke
the theory of $H^2$ strong solutions of the HJB equations.  In view of
the Cordes condition \eqref{eq:cordes-HJB}, for each
$\alpha \in \Lambda$,  define 
\begin{equation}\label{eq:gamma-HJB} 
\gamma^{\alpha} := \displaystyle\frac{ \operatorname{tr}A^{\alpha} +
  c^{\alpha}/\lambda }{|A^{\alpha}|^{2} +
  |\boldsymbol{b}^{\alpha}|^{2}/2\lambda + (c^{\alpha}/\lambda)^{2} }. 
\end{equation} 
And for $ \lambda $ as in \eqref{eq:cordes-HJB},  define a linear
operator $ L_{\lambda}: H^{2}(\Omega) \to L^{2}(\Omega) $ by 
\begin{equation}\label{eq:L-lambda} 
L_{\lambda} u: = \Delta u - \lambda u \qquad u \in H^{2}(\Omega). 
\end{equation} 
Next, we define the operator $F_{\gamma}: H^{2}(\Omega) \to
L^{2}(\Omega)$ by 
\begin{equation}\label{eq:F-gamma} 
F_{\gamma}[u] := \sup_{\alpha \in \Lambda} \{ \gamma^{\alpha}
  L^{\alpha}u - \gamma^{\alpha} f^{\alpha} \}.
\end{equation} 
Note that the continuity of data implies $\gamma^{\alpha} \in
C(\overline{\Omega}\times \Lambda; \mathbb{R})$. As a consequence, it
is readily seen that the HJB equation \eqref{eq:HJB} is equivalent to the problem $
F_{\gamma}[u] = 0$ in $\Omega$, and $u = 0$ on $\partial
\Omega$. The Cordes condition leads to the following lemma; See
\cite[Lemma 1]{smears2014discontinuous} for a proof. 

\begin{lemma}[property of Cordes condition]\label{lm:cordes-HJB-lemma}
  Under the Cordes condition \eqref{eq:cordes-HJB}, for any open set $
  U \subset \Omega$ and $ w,v \in H^{2}(U)$, $ z = w -v $, the
  following inequality holds a.e. in $ U $:
\begin{equation}\label{eq:cordes-HJB-lemma} 
  |F_{\gamma}[w] - F_{\gamma}[v] - L_{\lambda}z| \leq \sqrt{1 -
  \varepsilon} \sqrt{|D^{2} z |^{2} + 2\lambda|\nabla z|^{2} +
  \lambda^{2} z^{2}}. 
\end{equation} 
\end{lemma} 

Another key ingredient for the well-posedness of \eqref{eq:HJB} is the
Miranda-Talenti estimate stated as follows. 
\begin{lemma}[Miranda-Talenti estimate, \cite{grisvard2011elliptic,
maugeri2000elliptic}]\label{lm:M-T}
  Suppose $ \Omega $ is a bounded convex domain in $ \mathbb{R}^{d} $.
  Then, for any $ v \in V = H^{2}(\Omega) \cap H^{1}_{0}(\Omega) $, 
\begin{equation}\label{eq:M-T} 
 | v |_{H^{2}(\Omega)} \leq \|\Delta v\|_{L^{2}(\Omega)} \leq C | v
  |_{H^{2}(\Omega)},  
\end{equation} 
where the constant $C$ depends only on the dimension. 
\end{lemma} 

Let the operator $ M:V \to V^{\prime} $ be 
\begin{equation}\label{def:HJB-var} 
\langle M[w], v \rangle_{} := \displaystyle\int_{\Omega}^{}
  F_{\gamma}[w] L_{\lambda}v \mathrm{d}x. 
\end{equation} 
By using Miranda-Talenti estimate \eqref{eq:M-T} and Cordes condition
\eqref{eq:cordes-HJB}, one can show the strong monotonicity of $ M$,
\begin{equation*} \label{eq:HJB-monotonicity}
\langle M[v], v \rangle \geq (1 - \sqrt{1-\varepsilon}) \|v\|_{\lambda}^2
  \qquad \forall v \in V,
\end{equation*}
where $\|v\|_\lambda^2 := \|D^2 v\|_{L^2(\Omega)}^2 + 2\lambda\|\nabla
v\|_{L^2(\Omega)}^2 + \lambda^2\|v\|_{L^2(\Omega)}^2$.
Together with the Lipschitz continuity of $M$, the compactness of
$\Lambda$ and the Browder-Minty Theorem \cite[Theorem
10.49]{renardy2006introduction}, one can show the existence and
uniqueness of the following problem: Find $u \in V$ such that  
\begin{equation}\label{eq:HJB-var}
\langle M[u], v \rangle = 0 \qquad \forall v \in V.
\end{equation}
We refer to \cite[Theorem 3]{smears2014discontinuous} for a detailed
proof. 

\subsection{\texorpdfstring{$C^0$}{C0} finite element approximations
of the HJB equations}
Let $ \mathcal{T}_{h} $ be a conforming shape regular simplicial
triangulation of polytope $\Omega$ and $\mathcal{F}_{h}$ be the set of
all faces of $\mathcal{T}_{h}$. $\mathcal{F}^{i}_{h} :=
\mathcal{F}_{h} \backslash \partial \Omega$ and $
\mathcal{F}^{\partial}_{h} := \mathcal{F}_{h} \cap \partial \Omega$.
Let $ \mathcal{N}_{h} $ be the set of all the nodes of $
\mathcal{T}_{h} $. Here $ h = \max_{T \in \mathcal{T}_{h}} h_{T} $,
where $ h_{T} $ is the diameter of $ T \in \mathcal{T}_{h} $. We also
denote $ h_{F} $ as the diameter of $ F \in \mathcal{F}_{h} $. For  $
F \in \mathcal{F}_{h} $ and $ T \in \mathcal{T}_{h} $, we use $ (
\cdot, \cdot )_{T}  $ , respectively $ \langle \cdot, \cdot
\rangle_{F}  $, to denote the $ L^{2} $-inner product over $ T $,
respectively $ F $. 

Following \cite{wu2021c0}, we adopt the $ \mathcal{P}_{k} $-Hermite
finite elements ($ k \geq 3 $)  in 2D and $ \mathcal{P}_{k} $-Argyris
finite elements in 3D to solve the HJB equations \eqref{eq:HJB}.
Define the finite element spaces $ V_{h} $ as 
    \begin{enumerate}
    \item For $d = 2$, with $k \geq 3$ (cf. Fig. \ref{fig:Hermite-2D}), 
    \begin{equation*} \label{eq:FEM-Hermite}
    V_h := \{v \in H_0^1(\Omega): v|_T \in \mathcal{P}_k(T), \forall T \in
    \mathcal{T}_h, ~v \text{ is }C^1 \text{ at all vertices}\},
    \end{equation*}
    \item For $d=3$, with $k \geq 5$ (cf. Fig. \ref{fig:Argyris-3D}), 
    \begin{align} \label{eq:FEM-Argyris}
    V_h := \{v \in H_0^1(\Omega): v|_T \in \mathcal{P}_k(T), \forall T \in
    \mathcal{T}_h, ~& v \text{ is
    }C^1 \text{ on all edges}, \nonumber \\
    & v \text{ is }C^2 \text{ at all vertices}\}, \nonumber
    \end{align}
    \end{enumerate}
where $ \mathcal{P}_{k}(T) $ denotes set of the polynomials of degree
$k$ on $T$. 
\begin{figure}[!htbp]
  \centering 
  \captionsetup{justification=centering}
  \subfloat[2D Hermite element, $k=3$]{
    \includegraphics[width=0.24\textwidth]{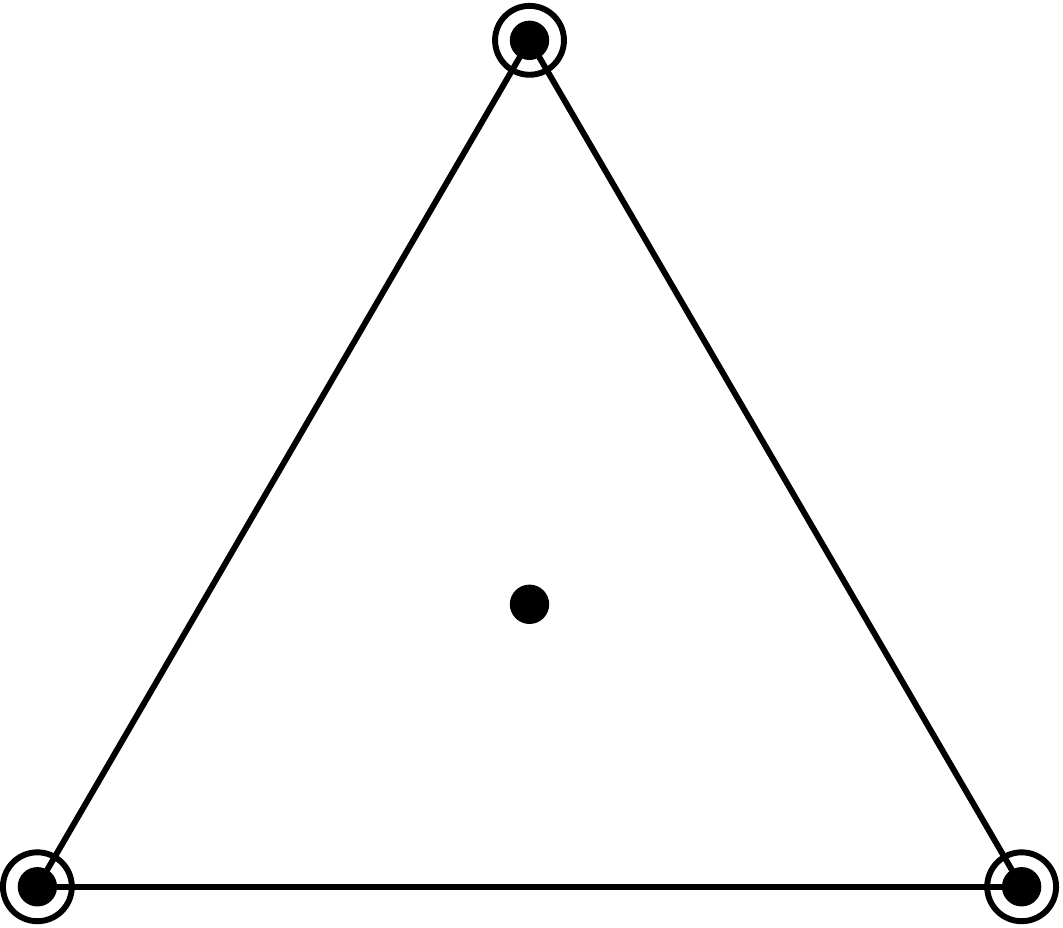}
    \label{fig:Hermite-2Dk3}
  }\qquad %
  \subfloat[2D Hermite element, $k=4$]{
    \includegraphics[width=0.24\textwidth]{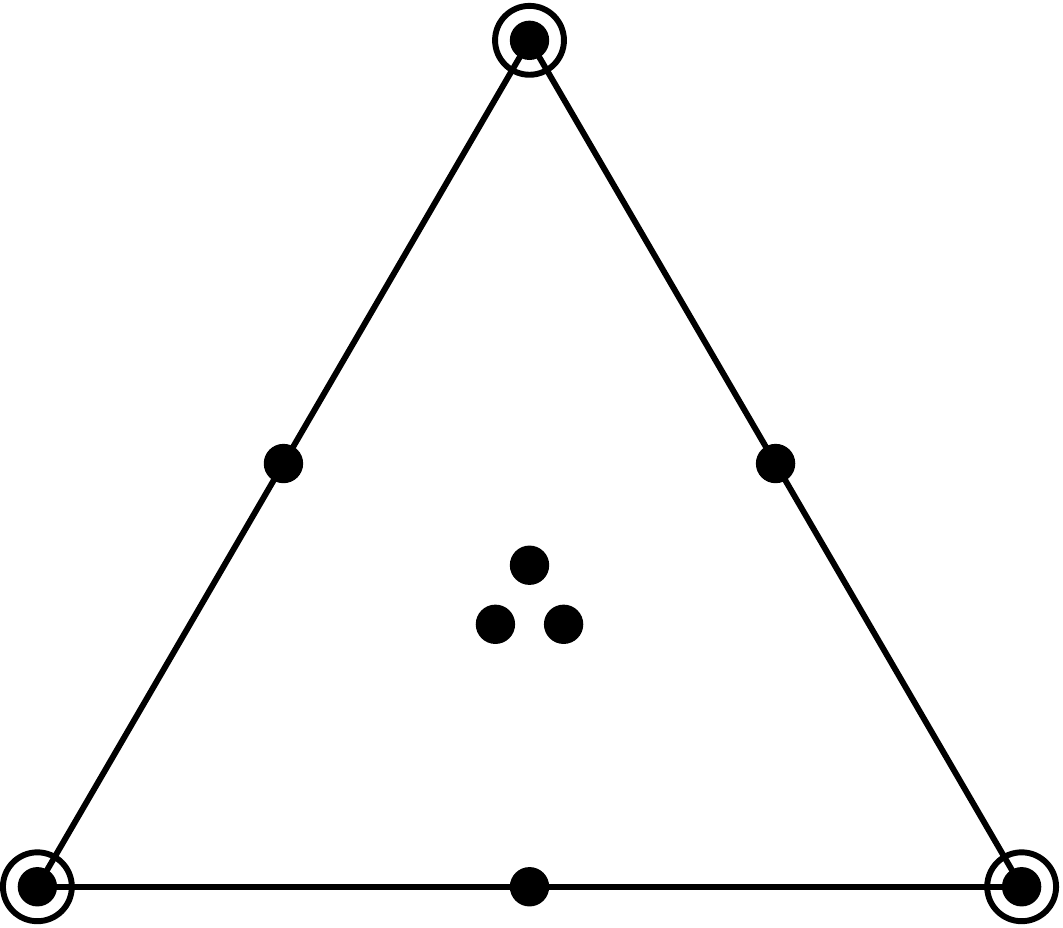}
    \label{fig:Hermite-2Dk4}
  } 
  \caption{Degrees of freedom of 2D $\mathcal{P}_k$ Hermite elements, in
    the case of $k=3$ and $k=4$}
  \label{fig:Hermite-2D}
  \end{figure}
  
  \begin{figure}[!htbp]
    \centering 
    \captionsetup{justification=centering}
    \subfloat[3D Argyris elements, $k=5$]{
      \includegraphics[width=0.28\textwidth]{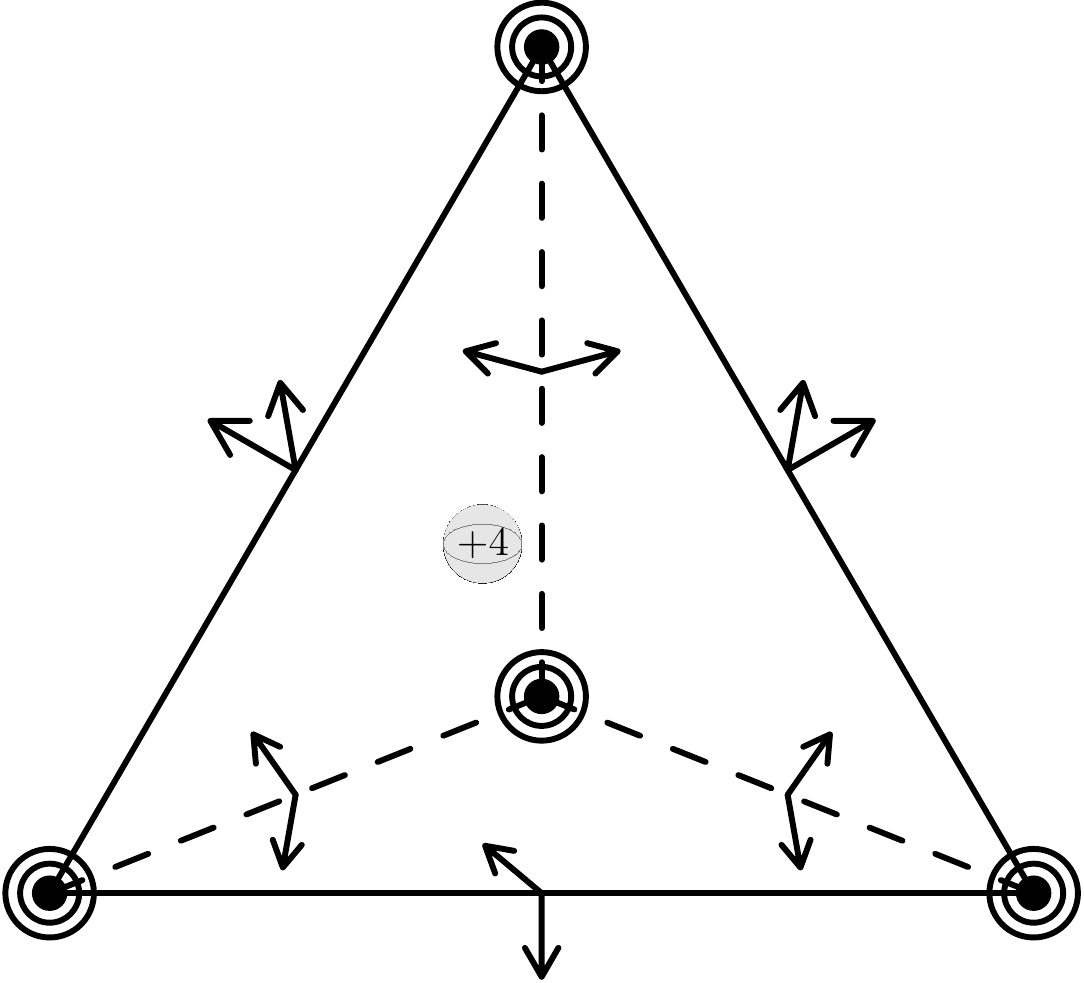}
      \label{fig:Argyris-3Dk5}
    }\qquad %
    \subfloat[3D Argyris elements, $k=6$]{
      \includegraphics[width=0.28\textwidth]{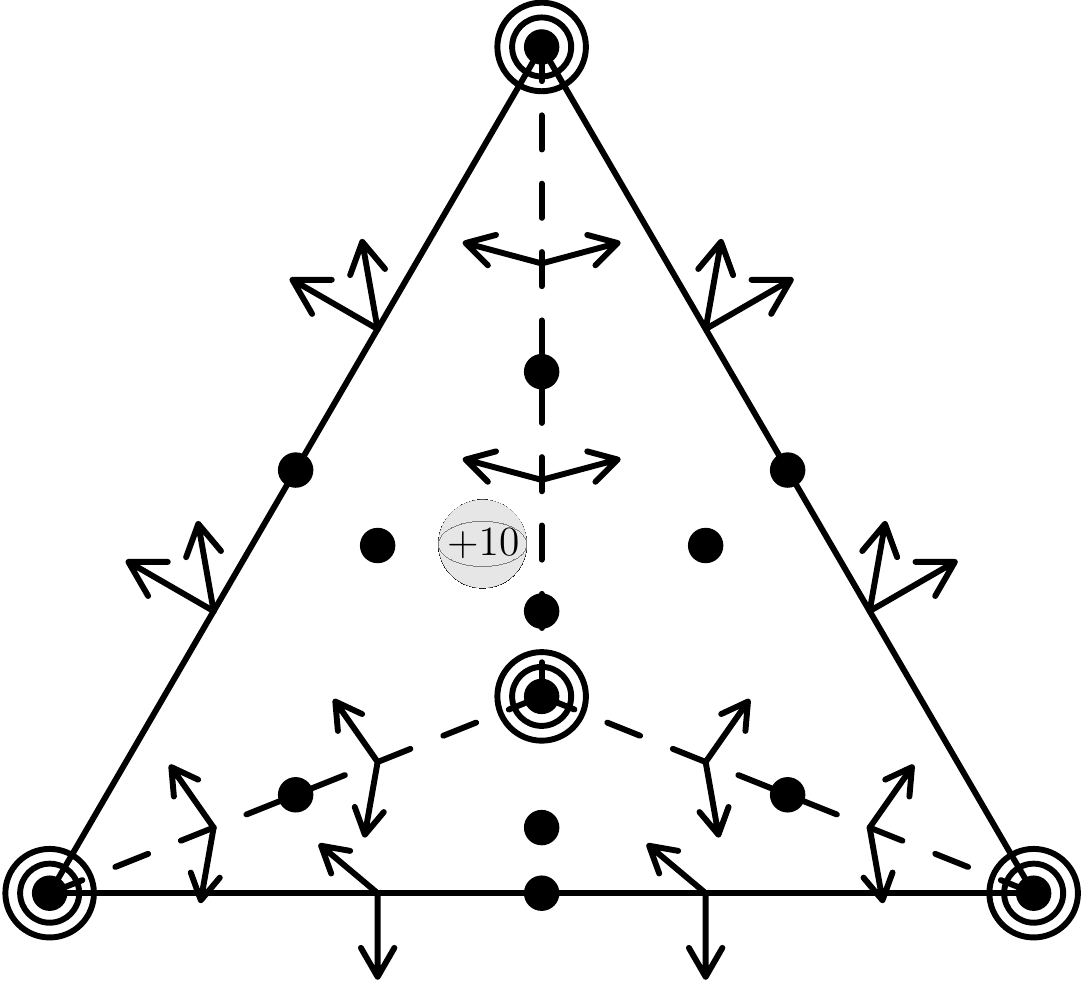}
      \label{fig:Argyris-3Dk6}
    } 
    \caption{Degrees of freedom of 3D $\mathcal{P}_k$ Argyris elements, in
      the case of $k=5$ and $k=6$}
    \label{fig:Argyris-3D}
    \vspace{-4mm}
    \end{figure}
    
For each $F\in\mathcal{F}^{i}_{h}$, we define the tangential
Laplace operator $ \Delta_{T}:H^{s}(F) \to H^{s-2}(F) $ as
follows, where $ s \geq 2$. Let $ \{\boldsymbol{t}_{i}\}^{d-1}_{i = 1}
$ be a orthogonal coordinate system on $F$. Then, for $ w \in
H^{s}(F) $ define
\begin{equation*}\label{eq:tan-laplace} 
\Delta_{T} w = \displaystyle\sum_{i = 1}^{d-1} \frac{\partial
  ^{2}}{\partial \boldsymbol{t}_{i}^{2}} w. 
\end{equation*} 
Next, we define the jump of a vector function $ \boldsymbol{v} $ on an
interior face $ F = \partial T^{+} \cap \partial T^{-} $  as follows: 
\begin{equation*}\label{def:jump-vector} 
    \left. \llbracket \boldsymbol{v} \rrbracket \right|_{F} := \boldsymbol{v}^{+}\cdot\boldsymbol{n}^{+}|_{F} + \boldsymbol{v}^{-}\cdot\boldsymbol{n}^{-}|_{F} ,
\end{equation*} 
where $ \boldsymbol{v}^{\pm} = \boldsymbol{v}|_{T^{\pm}} $ and $
\boldsymbol{n}^{\pm} $ is the unit outward normal vector of $ T^{\pm}
$, respectively. For scaler function $ w $ we define 
\begin{equation*}\label{def:jump-scaler} 
\llbracket w \rrbracket := w|_{T^{+}} - w|_{T^{-}}. 
\end{equation*} 
The following lemma is critical in the design and analysis of 
finite element approximation of HJB equations \eqref{eq:HJB}. 

\begin{lemma}[discrete Miranda-Talenti identity,
  \cite{wu2021c0}]\label{lm:discrete-M-T}
  Let $ \Omega \subset \mathbb{R}^{d} $ be a convex polytopal domain
  and $ \mathcal{T}_{h} $ be a conforming triangulation. For each $
  v_{h} \in V_{h} $, it holds that
\begin{equation*}\label{eq:discrte-M-T} 
\displaystyle\sum_{T \in \mathcal{T}_{h}}^{}\|\Delta
  v_{h}\|^{2}_{L^{2}(T)} = \displaystyle\sum_{T \in
  \mathcal{T}_{h}}^{} \|D^{2} v_{h}\|^{2}_{L^{2}(T)} +
  2\displaystyle\sum_{F \in \mathcal{F}^{i}_{h}}^{}\langle \llbracket
  \nabla v_{h} \rrbracket, \Delta_{T} v_{h} \rangle_{F}.  
\end{equation*} 
\end{lemma} 

In light of \eqref{def:HJB-var}, we define the operator $ M_{h}: V +
V_{h} \to V^{\prime}_{h} $ by 
\begin{equation*}\label{eq:HJB-variation-FEM} 
 \begin{aligned}
  \langle M_{h}[w], v_{h}\rangle : = & \displaystyle\sum_{T \in
   \mathcal{T}_{h}}^{}(F_{\gamma}[w], L_{\lambda} v_{h})_{T} \\ 
   & - (2 -
   \sqrt{1-\varepsilon})\displaystyle\sum_{F \in
   \mathcal{F}^{i}_{h}}^{} \langle \llbracket \nabla w \rrbracket,
   \Delta_{T} v_{h} - \lambda v_{h} \rangle_{F} .
 \end{aligned} 
\end{equation*} 
The following finite element scheme is proposed to approximate the
solutions to the HJB equations \eqref{eq:HJB}: Find : $ u_{h} \in
V_{h}$ such that 
\begin{equation}\label{eq:HJB-fem} 
\langle M_{h} u_{h}, v_{h} \rangle_{} = 0 \qquad \forall v_{h} \in V_{h}. 
\end{equation} 
We refer to \cite{wu2021c0} for the well-posedness and approximation
property of discrete systems \eqref{eq:HJB-fem}. 

\subsection{Semi-smooth Newton method}
It is shown in \cite{smears2014discontinuous} that the discretized
nonlinear system \eqref{eq:HJB-fem} can be solved by a semi-smooth
Newton method, which leads to a sequence of discretized linear
systems. We summarized the main ideas on semi-smooth Newton here and
refer \cite{smears2014discontinuous} for more detials. 

Following the discuss in \cite{smears2014discontinuous}, we define the
admissible maximizers set for any $ v\in V_{h} + V $, 
\begin{equation*} \label{eq:maximizer-HJB}
    \Lambda[v] := 
    \left\{ 
    \parbox{5.7em}{
    $\alpha(\cdot): \Omega \to \Lambda$ \\ 
    measurable} 
    \Bigg|~ 
    \parbox{20em}{ 
    $
    \displaystyle \alpha(x) \in
    \mathop{\arg\max}_{\alpha\in\Lambda}(A^\alpha:D_h^2v + \bm{b}^\alpha
    \cdot \nabla v - c^\alpha v - f^\alpha)$ \\
    for almost every $x\in \Omega$
    }
    \right\},
    \end{equation*}
where $ D_{h}^2 v $ denotes the broken Hessian of $ v $. As shown in
\cite[Lemma 9 \& Theorem 10]{smears2014discontinuous}, the set $
\Lambda[v] $ is not empty for any $ v \in V+V_{h} $. 

The semi-smooth Newton method is now stated as follows. Start by
choosing an initial iterate $ u^{0}_{h} \in V_{h} $. Then, for each
nonnegative integer $ j $, given the previous iterate $ u^{j}_{h} \in
V_{h} $, choose an $ \alpha_{j} \in \Lambda[u^{j}_{h}] $. Next the
function $ f^{\alpha_{j}}: \Omega \mapsto \mathbb{R} $ is defined by $
f^{\alpha_{j}}: x \to f^{\alpha_{j}(x)}(x) $; the functions $
A^{\alpha_{j}} $, $ \boldsymbol{b}^{\alpha_{j}} $, $ c^{\alpha_{j}} $
and $ \gamma^{\alpha_{j}} $ are defined in a similar way. Then find
the solution $ u^{j+1}_{h} \in V_{h} $ of the linearised system 
\begin{equation}\label{eq:linear-HJB-fem} 
    b_{\lambda,h}^{j}(u^{j+1}_{h}, v_{h}) = \displaystyle\sum_{T \in
    \mathcal{T}_{h}}^{} (\gamma^{\alpha_{j}}  f^{\alpha_{j}}, \Delta
    v_{h} )_{T} \qquad \forall v_{h} \in V_{h}, 
\end{equation}  
where the bilinear form $ b_{\lambda, h}^{j} : V_{h} \times V_{h}
\to \mathbb{R} $ is defined by 
\begin{equation*}\label{eq:linear-bilinear-HJB-fem} 
\begin{aligned}
        b^{j}_{\lambda, h}(w_{h}, v_{h})  := &
        \displaystyle\sum_{T
        \in\mathcal{T}_{h}}^{}(\gamma^{\alpha_{j}}
        L^{\alpha_{j}}w_{h}, L_{\lambda} v_{h})_{T} \\ &- (2 -
        \sqrt{1-\varepsilon})\displaystyle\sum_{F \in
        \mathcal{F}^{i}_{h}}^{} \langle  \llbracket \nabla w_{h}
        \rrbracket, \Delta_{T} v_{h} - \lambda v_{h} \rangle_{F}.  
\end{aligned}
\end{equation*}  
Following \cite{wu2021c0}, we define inner product on $ V_{h} $ as 
\begin{equation*}\label{def:inner-Vh} 
    (w_{h}, v_{h})_{\lambda, h} := \displaystyle\sum_{T \in
    \mathcal{T}_{h}}^{}(D^{2}w_{h}, D^{2}v_{h})_{T} + 2\lambda(\nabla
    w_{h}, \nabla v_{h})_{\Omega} + \lambda^{2} (w_{h},
    v_{h})_{\Omega}, 
    \end{equation*} 
and the norm $ \|v_{h}\|_{\lambda, h}^{2} := (v_{h},
v_{h})_{\lambda, h}  $. It is also shown in \cite{wu2021c0} that
the bilinear forms $ b_{\lambda, h}^{j} $ are uniformly coercive and
bounded on $ V_{h} $ with norm $ \|\cdot\|_{\lambda, h} $, with
constants independent of iterates. Since the preconditioners in this
work take advantage on the coercivity and boundedness of $
b^{j}_{\lambda, h} $, we summarize the relevant results in the
following lemma (see \cite[Lemmas 4.1 \& 4.2]{wu2021c0}  for a detailed
proof). 
\begin{lemma}[coercivity and boundedness of bilinear
  form]\label{lm:coer-bbd}For every $ w_{h}, v_{h} \in V_{h} $, we
  have 
\begin{subequations}\label{eq:coer-bdd}
  \begin{align}
    b_{\lambda, h}^{j}(v_{h}, w_{h}) & \leq C \|v_{h}\|_{\lambda, h}
        \|w_{h}\|_{\lambda, h}, \nonumber\\
    b_{\lambda, h}^{j}(v_{h}, v_{h}) &\geq  (1 - \sqrt{1 -
        \varepsilon}) \|v_{h}\|^{2}_{\lambda, h}. \nonumber  
  \end{align}
\end{subequations}    
Here, the constant $C$ depends only on $\Omega$, shape regularity
of the grid and polynomial degree $k$. 
\end{lemma} 

\section{FOV-equivalent preconditioners for
GMRES methods} \label{sec:fov}

The preconditioned GMRES (PGMRES) methods are among the most effective
iterative methods for non-symmetric linear systems arising from
discretizations of PDEs. Our study will start by discussing PGMRES
methods in an operator form.  Let $ G: \mathcal{X} \to
\mathcal{X}^{\prime} $ be  a linear operator which may be
non-symmetric or indefinite, defined on a finite dimensional space $
\mathcal{X} $, and $g$ be a given functional in its dual space $
\mathcal{X}^{\prime} $. The linear equation considered here is of the
following form 
\begin{equation}\label{eq:algebraic-equ}
Gx = g. 
\end{equation} 
Let $ (\cdot,\cdot)_{M}$ be an inner product on $ \mathcal{X}$,
and $P: \mathcal{X}^{\prime} \to \mathcal{X}$ be the
preconditioner. The PGMRES method for solving \eqref{eq:algebraic-equ}
is stated as follows: Begin with an initial gauss $ x_{0} \in
\mathcal{X} $ and denote $r_{0} = g - G x_{0} $ the initial residual, the 
$k$-th steps of PGMRES method seeks $x_{k}$ such that  
$$ 
x_{k} = \underset{ \tilde{x}_{k} \in \mathcal{K}_{k}(PG, Pr_{0}) +
x_{0}}{\operatorname{argmin}}\|PG(x - \tilde{x}_{k})\|_{M},  
$$    
where $ \mathcal{K}_{k}(PG,Pr_{0}) $ is the Krylov subspace of
dimension $k$ generated by $ PG $ and $ P r_{0} $.

In the semi-smooth Newton steps, the discrete linear equations
\eqref{eq:linear-HJB-fem} have a common form: Find $ u_{h} \in V_{h} $
such that 
\begin{equation}\label{eq:general-form-fem} 
b_{\lambda, h}(u_{h}, v_{h}) = f_{h}(v_{h}) \qquad \forall v_{h} \in
  V_{h},
\end{equation}  
where we shall omit to denote independence of the bilinear form
$b_{\lambda, h} $ and of the right-hand side $ f_{h} $ on the
iteration number of the semi-smooth Newton method. Define the operator
$B_{\lambda, h}: V_{h} \rightarrow V^{\prime}_{h}$ by 
\begin{equation}\label{eq:operator-B} 
\langle B_{\lambda, h} u_{h}, v_{h} \rangle_{} := b^{}_{\lambda,
  h}(u_{h}, v_{h}) \qquad \forall u_{h}, v_{h} \in V_{h},
\end{equation}
then the discrete system \eqref{eq:general-form-fem} can be written in
an operator form, namely 
\begin{equation}\label{eq:operator-form-fem} 
    B_{\lambda, h}u _{h} = f_{h}.
\end{equation} 
Moreover, a general operator $ P_{\lambda, h}: V^{\prime}_{h}
\rightarrow V_{h} $ is used to denote the preconditioner. 
Given an inner product $ (\cdot, \cdot)_{M_{\lambda, h}}$, we can
estimate the convergence rate of the PGMRES method. It is proved in
\cite{elman1982iterative,saad2003iterative} that if $u_{h}^{m}$ is the
$m$-iteration of PGMRES method and $u_{h}$ is the exact solution of
\eqref{eq:operator-form-fem}, then 
\begin{equation*}
\displaystyle\frac{\|P_{\lambda, h} B_{\lambda, h} (u_{h} - u_{h}^{m})
  \|_{M_{\lambda, h}}}{\|P_{\lambda, h} B_{\lambda, h} (u_{h} - u_{h}^{0})
  \|_{M_{\lambda, h}}} \leq \left(1 -
  \frac{\gamma^{2}}{\Gamma^{2}}\right)^{m/2}, 
\end{equation*}    
where 
\begin{equation}\label{eq:fov-parameter} 
\gamma \leq \displaystyle\frac{(v_{h}, P_{\lambda, h} B_{\lambda,
  h}v_{h})_{M_{\lambda, h}}}{(v_{h},v_{h})_{M_{\lambda, h}}}, \quad
  \displaystyle\frac{\|P_{\lambda, h} B_{\lambda, h} v_{h}\|_{M_{\lambda,
  h}} }{\|v_{h}\|_{M_{\lambda, h}}} \leq \Gamma \qquad \forall v_{h} \in V_{h}. 
\end{equation} 
Therefore, we conclude that as long as we find an operator $
P_{\lambda, h} $ and a proper inner product $(\cdot,
\cdot)_{M_{\lambda, h}}$ such that condition \eqref{eq:fov-parameter}
is satisfied with constants $\gamma$ and $\Gamma$ independent of the
discretization  parameter $ h $ and the Cordes condition parameter $
\lambda $, then $ P_{\lambda, h} $ is a uniform preconditioner for
GMRES method.  Such preconditioners are usually referred to as
\textit{FOV-equivalent preconditioners}. In what follows, we always
take 
\begin{equation*} \label{eq:P-M}
  P_{\lambda,h} \text{ to be an SPD operator,} \quad \text{and} \quad
  M_{\lambda, h} = P_{\lambda,h}^{-1}.
\end{equation*} 

Next, we give a general principle for constructing $P_{\lambda,h}$.
Define an SPD operator $A_{\lambda, h}: V_{h} \to V^{\prime}_{h}$
by 
\begin{equation}\label{def:inner-operator} 
    \langle A_{\lambda, h} w_{h}, v_{h} \rangle_{} := (w_{h},
    v_{h})_{\lambda, h} \qquad  \forall w_{h},v_{h} \in V_{h}.  
\end{equation} 
Recalling Lemma~\ref{lm:coer-bbd} (coercivity and boundedness of bilinear
form), $b_{\lambda, h}(\cdot, \cdot)$ is coercive and bounded on $
V_{h}$ with the inner product $ (\cdot,\cdot)_{\lambda, h}$.  It is
therefore that an efficient preconditioner for $ A_{\lambda, h}$ can
also be used as an FOV-preconditioner for the GMRES algorithm applied
to $ B_{\lambda, h}$, which is shown in the following lemma.

\begin{lemma}[FOV-equivalent preconditioner]\label{lm:FOV-norm-equivalent}
 Let $ A_{\lambda, h} $ and $ B_{\lambda, h} $ be the operators
  defined in \eqref{def:inner-operator} and \eqref{eq:operator-B},
  respectively. If
  an SPD operator $ P_{\lambda, h} : V^{\prime}_{h} \rightarrow V_{h}
  $ satisfies that
\begin{equation}\label{eq:norm-equivalence} 
    \alpha \langle P_{\lambda, h}^{-1} v_{h}, v_{h} \rangle \leq
    \langle A_{\lambda, h} v_{h}, v_{h} \rangle \leq \beta \langle
    P_{\lambda, h}^{-1} v_{h}, v_{h} \rangle \qquad \forall v_{h}
    \in V_{h},    
\end{equation} 
with constants $\alpha, \beta$ independent of both $\lambda$ and $h$,
then $P_{\lambda,h}$ is a uniform FOV-equivalent preconditioner of
$B_{\lambda,h}$.  
\end{lemma} 
\begin{proof}
From \eqref{eq:operator-B}, \eqref{def:inner-operator} and
Lemma~\ref{lm:coer-bbd} (coercivity and boundedness of bilinear
form), we see that for any $u_{h}, v_{h} \in V_{h}$ 
\begin{subequations}\label{eq:bounded-coercive}
\begin{align}
  (1 - \sqrt{1 - \varepsilon}) \langle A_{\lambda, h} u_{h},
  u_{h} \rangle_{} &\leq \langle B_{\lambda, h}u_{h}, u_{h}
  \rangle_{}  \label{eq:coercive-B} \\ 
  \langle B_{\lambda, h} u_{h}, v_{h}
  \rangle_{} &\leq C \langle A_{\lambda, h} u_{h}, u_{h}
  \rangle^{1/2}_{}  \langle A_{\lambda, h} v_{h}, v_{h}
  \rangle^{1/2}_{}, \label{eq:bounded-B}
\end{align}
\end{subequations}      
where $C$ is independent of both $\varepsilon$ and $\lambda$. 

Recalling $M_{\lambda,h} := P_{\lambda, h}^{-1}$, then for any $ u_{h}
  \in V_{h}$, we have 
\begin{equation*}\label{eq:FOV-Gamma} 
\begin{aligned}
\|P_{\lambda, h} B_{\lambda, h} u_{h}\|_{P^{-1}_{\lambda, h}}
&=  \sup_{v_{h} \in V_{h} , v_{h} \neq 0}
\displaystyle\frac{(P_{\lambda, h} B_{\lambda, h} u_{h},
v_{h})_{P^{-1}_{\lambda, h}}}{(v_{h}, v_{h})_{P^{-1}_{\lambda,
h}}^{1/2}}  \\ 
  &\leq  \beta^{1/2} \sup_{v_{h} \in V_{h} , v_{h} \neq 0}
  \displaystyle\frac{\langle B_{\lambda, h} u_{h}, v_{h} \rangle_{}
  }{\langle A_{\lambda,h}v_{h}, v_{h} \rangle_{}^{1/2}}  ~~~~(\mbox{by
  } \eqref{eq:norm-equivalence}) \\ 
  &\leq C \beta^{1/2} \langle A_{\lambda,h} u_{h}, u_{h}
        \rangle^{1/2}_{} ~~~~~~~~~~~~~~~~(\mbox{by }
        \eqref{eq:bounded-B}) \\ 
  &\leq C \beta \|u_{h}\|_{P^{-1}_{\lambda, h}},
  ~~~~~~~~~~~~~~~~~~~~~~~~~~~(\mbox{by }
  \eqref{eq:norm-equivalence})  
\end{aligned}
\end{equation*} 
which yields the second inequality of \eqref{eq:fov-parameter} with
$\Gamma = C\beta$. On the other side, we have 
\begin{equation*}\label{eq:FOV-gamma} 
\begin{aligned}
  (u_{h}, P_{\lambda, h} B_{\lambda, h} u_{h})_{P^{-1}_{\lambda, h}} &=
  \langle B_{\lambda, h} u_{h}, u_{h}\rangle_{} \\ 
    &\geq (1 - \sqrt{1 - \varepsilon}) \langle A_{\lambda, h} u_{h},
    u_{h} \rangle_{} ~~~~~~~~ (\mbox{by }
    \eqref{eq:coercive-B})\\ 
    &\geq (1 - \sqrt{1 - \varepsilon}) \alpha (u_{h},
    u_{h})_{P^{-1}_{\lambda, h}},~~~~~~ (\mbox{by }
    \eqref{eq:norm-equivalence})
    \end{aligned}
    \end{equation*} 
which yields the first inequality of \eqref{eq:fov-parameter} with $
\gamma = (1 - \sqrt{1 - \varepsilon}) \alpha$.
\end{proof} 
 

\section{Fast auxiliary space preconditioners} \label{sec:solver}

In this section, we construct both additive and multiplicative
auxiliary space preconditioners for SPD operator $A_{\lambda, h}$.
From Lemma \ref{lm:FOV-norm-equivalent} (FOV-equivalent
preconditioner), those preconditioners can be applied to the discrete
linearised systems \eqref{eq:linear-HJB-fem} arising from each
semi-smooth Newton step of solving the HJB equations. 

\subsection{Space decomposition} 
For the purpose of constructing auxiliary space preconditioners, we
give the following space decomposition of $V_{h}$ as 
\begin{equation*}\label{eq:space-decom} 
V_{h} = V_{h} + \Pi_{0} V_{0},
\end{equation*} 
where the auxiliary space $V_{0} \subset H_0^{1}(\Omega)$ denotes the
continuous piecewise linear element space on $\mathcal{T}_{h}$ with
homogeneous Dirichlet boundary condition, and $\Pi_{0}: V_{0} \to
V_{h}$ is a linear injective map which will be defined later. As a
result, the induced operator  $A_{0} := \Pi_{0}^{\prime} A_{\lambda,
h} \Pi_{0}: V_{0}\to V^{\prime}_{0}$ is also SPD, and hence we define
$\|\cdot\|_{A_{0}}^{2} := \langle A_{0}\cdot,\cdot \rangle_{}$ on
$V_{0}$. We also introduce a projection $P_{0}: V_{h} \to V_{0}$
by 
\begin{equation*}\label{eq:projection-V0} 
\langle A_{0}P_{0} v_{h}, w_{h} \rangle_{} := (v_{h},
  \Pi_{0}w_{h})_{\lambda, h} \quad \forall v_{h} \in V_{h}, w_{h} \in
  V_{0}.  
\end{equation*} 
A direct calculation shows the following identity 
\begin{equation}\label{eq:A0P0} 
\Pi_{0}^{\prime} A_{\lambda, h} = A_{0} P_{0}. 
\end{equation} 

\paragraph{Smoother and norm on $V_0$} Define the discrete Laplacian operator $-\Delta_{h}: V_{0} \to
V_{0}$ by
\begin{equation}\label{eq:laplace-VL} 
(-\Delta_{h} w_{h}, v_{h} )_{\Omega} := (\nabla w_{h}, \nabla
  v_{h})_{\Omega} \qquad \forall w_{h}, v_{h} \in V_{0}. 
\end{equation} 
Then, the smoother on $V_{0}$, denoted by $R_{0} :
V^{\prime}_{0}\to V_{0}$, is defined by 
\begin{equation}\label{def:smoother-VL} 
    \langle R^{-1}_{0} u_{h}, v_{h} \rangle_{} = (\lambda u_{h} -
    \Delta_{h} u_{h}, \lambda v_{h} - \Delta_{h} v_{h})_{\Omega}
    \qquad \forall u_{h}, v_{h} \in V_{0}. 
\end{equation} 
Note that for any given $f_{h} \in V^{\prime}_{0}$, $u_{h} =
R_{0}f_{h} \in V_{0}$ can be obtained by solving the following two
discrete Poisson-like equations 
\begin{subequations} \label{eq:get-Rf} 
\begin{align}
\langle f_{h}, v_{h} \rangle_{} &= \lambda (z_{h},  v_{h})_{\Omega} +
  (\nabla z_{h}, \nabla v_{h})_{\Omega} \qquad \forall v_{h} \in V_{0},
  \label{eq:get-Rf1}\\
  (z_{h}, w_{h})_{\Omega} &= \lambda( u_{h}, w_{h})_{\Omega} + (
  \nabla u_{h}, \nabla w_{h})_{\Omega} \qquad \forall w_{h} \in V_{0}.
  \label{eq:get-Rf2}
\end{align}
\end{subequations}
It can be shown that the above two equations can be solved within
$\mathcal{O}(N \log N)$ operations, where $N$ denotes the number of
degrees of freedom. We will give a detailed explanation in Remark
\ref{subsec:complexity} (computational complexity). 
The smoother $R_0$ induces a norm on $V_0$, i.e.,
$\|\cdot\|^{2}_{R_{0}^{-1}} := \langle R_{0}^{-1} \cdot,\cdot
\rangle$. By using \eqref{eq:laplace-VL} and \eqref{def:smoother-VL},
it is straightforward to show that  
\begin{equation*}\label{eq:norm-VL} 
\begin{aligned}
    \|v_{h}\|_{R_{0}^{-1}}^{2} 
         =& \|\Delta_{h} v_{h}\|^{2}_{L^{2}(\Omega)} + 2\lambda
         \|\nabla v_{h}\|^{2}_{L^{2}(\Omega)} + \lambda^{2}
         \|v_{h}\|^{2}_{L^{2}(\Omega)}.  \\
\end{aligned}
\end{equation*} 
The relationship between $\|\cdot\|_{R_{0}^{-1}}$ and
$\|\cdot\|_{A_{0}}$ is shown in the following lemma, whose proof is
postponed to Section \ref{sec:analysis}. 
\begin{lemma}[spectral equivalence of $ R_{0}
  $]\label{lm:spectral-equi-R0}
    Let $R_{0}$ be the operator defined in \eqref{def:smoother-VL} and
    $A_{0} = \Pi_{0}^{\prime} A_{\lambda, h} \Pi_{0}$. Then,  
    \begin{equation*}
    \|v_{0}\|_{R_{0}^{-1}} \simeq \|v_{0}\|_{A_{0}} \qquad \forall v_{0}
     \in V_{0},
    \end{equation*} 
 with hidden constants independent of both $\lambda$ and $h$.    
\end{lemma}

\paragraph{Smoother on $V_h$} Let ${R}_{h}$ denote the Gauss-Seidel
smoother for $A_{\lambda, h}$ and $ \bar{R}_{h}$ be the symmetric
Gauss-Seidel smoother, i.e.,
$$
I - \bar{R}_h A_{\lambda,h} = (I-R_h'A_{\lambda,h})(I - R_h
A_{\lambda, h}).
$$
We also define $\|\cdot\|_{\bar{R}_{h}^{-1}}^{2} := \langle
\bar{R}^{-1}_{h} \cdot,\cdot  \rangle$ as a norm on $V_{h}$.  Note
that $A_{\lambda, h}$ is an SPD operator, thus $R_{h}$ has the
following contraction property 
\begin{equation}\label{eq:contraction-GS} 
\|I - R_{h}A_{\lambda, h}\|_{\lambda, h} < 1.
\end{equation} 

\paragraph{Transfer operator} We now give the definition of
$\Pi_{0} :V_{0} \to V_{h}$. To this end, we first give some
notation on the degrees of freedom of finite element space $V_{h}$.
We denote the degrees of freedom as 
\begin{equation*}\label{eq:dof-Hermite}
    \mathcal{N}_{\alpha}(\varphi) = \fint_{D_{\alpha}} 
    \nabla^{k_{\alpha} }(\varphi) (\boldsymbol{t}_{1}, \ldots,
     \boldsymbol{t}_{k_{\alpha}}) 
\end{equation*}
where $D_{\alpha}$ is the domain of the integral with respect to 
the degree of freedom, $\fint_{D_{\alpha}}$ denotes the 
integral average on $D_{\alpha}$. In general, $D_{\alpha}$ is a 
subsimplex of the triangulation. When $D_{\alpha}$ is a point, 
the average of the integral is reduced to the evaluation on the
point. 
$\boldsymbol{t}_{1}, \cdots, \boldsymbol{t}_{k_{\alpha}}$ are
$k_{\alpha}$ identical or different unit vectors to denote the
direction of the derivative where $k_{\alpha} = 0,1,2$. When 
$ k_{\alpha} = 1 $ only one direction is involved for the 
derivative, and the direction is denoted by $\boldsymbol{t}_{\alpha}$.
Let $\varphi_{\alpha}$ be the nodal basis function corresponding to
 $\mathcal{N}_{\alpha}$. Define $\omega_{\alpha} := \bigcup \{ T: \mathring{T} \cap
\operatorname{supp}(\varphi_{\alpha}) \neq  \varnothing, T \in \mathcal{T}_{h}\}$, $\#
\omega_{\alpha}: = \#\{ T : \mathring{T} \cap
\operatorname{supp}(\varphi_{\alpha}) \neq  \varnothing, T \in \mathcal{T}_{h}\}$ and
$h_{\alpha} = \max_{T \subset \omega_{\alpha}} h_{T}$. We are now
ready to give the definition of $\Pi_{0}$ as follows: For any $p_{h}
\in V_{0}$
\begin{equation}\label{def:Pi_h} 
\begin{aligned}
\mathcal{N}_{\alpha}(\Pi_{0} p_{h}) &= \displaystyle\frac{1}{\# \omega_{\alpha}} \displaystyle\sum_{T \subset \omega_{\alpha}}^{} \fint_{D_{\alpha}} \partial_{\boldsymbol{n}}(\left. p_{h}\right|_{T})(\boldsymbol{n}\cdot \boldsymbol{t}_{{\alpha}}) \\
&\text{when }D_{\alpha} \subset \partial \Omega\text{ and }k_{\alpha} = 1, \\
\mathcal{N}_{\alpha}(\Pi_{0} p_{h}) &= \displaystyle\frac{1}{\# \omega_{\alpha}} \displaystyle\sum_{T \subset \omega_{\alpha}}^{} \mathcal{N}_{\alpha}(\left. p_{h}\right|_{T}), \text{ else}, 
\end{aligned}
\end{equation} 
where $\boldsymbol{n}$ is the unit outer normal vector of $\partial
\Omega$. We note that the degrees of freedom corresponding to the
second-order derivative vanish since $p_h$ is piecewise linear.

The general theory of auxiliary space preconditioning simplifies the
analysis of preconditioners to the verification of the following two
key assumptions.

\begin{assumption}[stable decomposition]\label{prop:stab-decom}
There exists a uniform constant $c_{0}$ independent of both $\lambda$
  and $h$, such that for any $v \in V_{h}$, there exist $v_{h} \in
  V_{h}$ and $v_{0} \in V_{0}$ satisfy  
\begin{subequations} \label{eq:stab-decom} 
    \begin{align} 
      v &= v_{h} + \Pi_{0}v_{0}, \label{eq:stab-decom1} \\
      \|v_{h}\|^{2}_{\bar{R}^{-1}_{h}} + \|v_{0}\|_{R_{0}^{-1}}^{2}
        & \leq c_{0}^{2}  \|v\|_{\lambda, h}^{2}. \label{eq:stab-decom2}
    \end{align}
    \end{subequations}
\end{assumption}    

\begin{assumption}[boundedness]\label{prop:boundedness}
There exist uniform constants $ c_{1}$  and $ c_{2} $ independent of
  both $\lambda $ and $ h $ such that 
\begin{subequations} \label{eq:boundedness} 
\begin{align} 
    \|v_{0}\|_{A_{0}} &\leq c_{1} \|v_{0}\|_{R_{0}^{-1}} \quad \forall
    v_{0} \in V_{0}, \label{eq:boundedness1} \\
    \|v_{h}\|_{\lambda, h} &\leq c_{2}  \|v_{h}\|_{\bar{R}^{-1}_{h}}
    \quad \forall v_{h} \in V_{h}. \label{eq:boundedness2} 
\end{align}
\end{subequations}
\end{assumption}


\subsection{Additive preconditioner}
Firstly, we introduce the additive preconditioner $P_{\rm{a}}
  :V_{h}^{\prime} \to V_{h}$ as 
\begin{equation}\label{def:add-precond} 
    P_{\rm{a}}: = \bar{R}_{h} + \Pi_{0} R_{0} \Pi_{0}^{\prime}.
\end{equation} 
The following theorem plays a fundamental role in the theory of
  auxiliary space preconditioning \cite{xu1996auxiliary}.  
\begin{theorem}[spectral equivalence of additive
  preconditioner]\label{thm:add-precond}
    Let $P_{\rm{a}}: V_{h}^{\prime} \to V_{h}$ be the preconditioner
    defined in \eqref{def:add-precond}. If Assumption
    \ref{prop:stab-decom} (stable decomposition) and Assumption
    \ref{prop:boundedness} (boundedness) hold, then we have 
\begin{equation*}\label{eq:add-equivalence} 
c_{0}^{-2} (v_{h}, v_{h})_{\lambda, h} \leq (P_{\rm{a}} A_{\lambda, h}
  v_{h}, v_{h} )_{\lambda, h} \leq (c_{1}^{2} + c_{2}^{2}) (v_{h},
  v_{h})_{\lambda, h} \quad \forall v_{h} \in V_{h}, 
\end{equation*}  
That is, $P_{\rm{a}}$ is a uniform spectral equivalence preconditioner
of $A_{\lambda, h}$. 
\end{theorem} 

In light of the above theorem and Lemma \ref{lm:FOV-norm-equivalent} (FOV-equivalent preconditioner),
one can see that $P_{\rm{a}}$ is a uniform FOV-equivalent
preconditioner of $ B_{\lambda, h}$ as long as Assumption
\ref{prop:stab-decom} (stable decomposition) and Assumption
\ref{prop:boundedness} (boundedness) are verified. We postpone those
verifications to Section \ref{sec:analysis}. 

\begin{remark}[additive preconditioner with Jacobi smoother] \label{rmk:add-Jacobi}
Let $D_h^{-1}$ be the Jacobi smoother of $A_{\lambda,h}$. From the norm equivalence between $ \bar{R}_{h}$ and $D_h^{-1}$ \cite[Lemma 4.6]{xu2017algebraic}, the additive preconditioner 
$$ 
\tilde{P}_{\rm{a}} := D_h^{-1} + \Pi_{0} R_{0} \Pi_{0}^{\prime}
$$ 
is also a uniform spectral equivalence preconditioner of $A_{\lambda,h}$.
\end{remark}

\begin{remark}[additive preconditioner with scaled
parameter]\label{rmk:omega-additive}
When implementing the additive preconditioners, a positive parameter
$\omega $ is usually introduced to balance the two components, namely, 
\begin{equation}\label{eq:omega-additive} 
    P_{\rm{a}} := \bar{R}_{h} + \omega \Pi_{0} R_{0} \Pi_{0}^{\prime}. 
\end{equation} 
A proper choice of $ \omega $ may lead to a better preconditioning
performance of $P_{\rm{a}}$ in practice.  
\end{remark}

\subsection{Multiplicative preconditioner}
We introduce the multiplicative preconditioner $ P_{\rm{m}}:
  V^{\prime}_{h} \to V_{h} $ by 
\begin{equation}\label{def:multi-precond} 
 I - P_{\rm{m}} A_{\lambda, h} := (I - R_{h} A_{\lambda, h})(I -
  \Pi_{0}R_{0}\Pi_{0}^{\prime}A_{\lambda, h})(I - R_{h}'A_{\lambda,
  h}).
\end{equation} 

Let $ \hat{R}_{0} = A_{0}^{-1}$ be the exact solver on $ V_{0}$. 
To analyse the multiplicative preconditioner $P_{\rm{m}}$, we
introduce an auxiliary multiplicative preconditioner $\hat{P}_{\rm{m}}
  : V_{h}^{\prime} \to V_{h} $ by 
\begin{equation}\label{eq:multi-precond-hat} 
    I - \hat{P}_{\rm{m}} A_{\lambda, h} := (I - R_{h}
    A_{\lambda, h})(I - \Pi_{0}\hat{R}_{0}\Pi_{0}^{\prime}A_{\lambda,
    h})(I - R_{h}'A_{\lambda, h}).
\end{equation} 
We emphasis that $\hat{P}_m$ is never computed but is useful for
theoretical purposes in Theorem \ref{thm:multi-precond} (spectral equivalence of multiplicative
preconditioner), which can be
divided into two steps: (i) The spectral equivalence between
$P_mA_{\lambda,h}$ and $\hat{P}_m A_{\lambda,h}$ by using Lemma
\ref{lm:spectral-equi-R0} (spectral equivalence of $ R_{0}
$); (ii) Estimate of $\hat{P}_m A_{\lambda,h}$
by two-level convergence results. For the second step, let $E := (I -
R_{h} A_{\lambda, h})(I -
\Pi_{0}\hat{R}_{0}\Pi_{0}^{\prime}A_{\lambda, h}) $ be the error
propagation operator the two-level method corresponding to
$\hat{P}_m$. 
Since the solver on $V_0$ is exact, the convergence rate
can be obtained in the following theorem. We refer to \cite[Theorem
5.3]{xu2017algebraic} for more details. 
\begin{theorem}[two-level convergence rate,
  \cite{xu2017algebraic}]\label{thm:two-level}
The following identity holds 
\begin{equation*}\label{eq:X-Z identity} 
\|E\|_{\lambda, h}^{2} = 1 - \displaystyle\frac{1}{K(V_{0})},
\end{equation*} 
where 
\begin{equation*}\label{eq:KV0} 
K(V_{0}) := \max_{v \in V_{h}} \min_{v_{0}\in V_{0}}
  \displaystyle\frac{\|v - \Pi_{0}
  v_{0}\|_{\bar{R}_{h}^{-1}}^{2}}{\|v\|_{\lambda, h}^{2}}. 
\end{equation*} 
\end{theorem} 

Note that the identity \eqref{eq:A0P0} implies that $\hat{R}_0
\Pi_0'A_{\lambda,h} = P_0$. For any $v_h, w_h \in V_h$, we have  
$$ 
\begin{aligned}
  (\Pi_0 P_0 v_h, w_h)_{\lambda,h} &= \langle P_0 v_h,
  \Pi_0'A_{\lambda,h}w_h \rangle = \langle P_0 v_h, A_0 P_0 w_h
  \rangle = (v_h, \Pi_0 P_0 w_h)_{\lambda,h}, \\
  (R_h' A_{\lambda,h}v_h, w_h)_{\lambda,h} &= \langle
  A_{\lambda,h}v_h, R_hA_{\lambda,h}w_h \rangle = (v_h,
  R_hA_{\lambda,h}w_h)_{\lambda,h},
\end{aligned}
$$ 
which means that $I - \Pi_{0}P_{0}$ and $I - R_h'A_{\lambda,h}$ are respectively
the dual operators of $I - \Pi_0 P_0$ and $I - R_hA_{\lambda,h}$ under the inner
product $(\cdot,\cdot)_{\lambda, h}$. As a consequence, 
\begin{equation} \label{eq:X-Z-dual}
  \|E\|_{\lambda,h}^2 = \|(I - R_hA_{\lambda,h})(I - \Pi_0
  P_0)\|_{\lambda,h}^2 = \|(I - \Pi_0
  P_0)(I - R_h'A_{\lambda,h})\|_{\lambda,h}^2.
\end{equation}
Moreover, under the Assumption \ref{prop:stab-decom} (stable
decomposition), we have $ K(V_{0}) \leq c_{0}^{2} $ and hence
$\|E\|_{\lambda, h}^{2} \leq 1 - \displaystyle\frac{1}{c_{0}^{2}}$,
which leads to the following theorem.

\begin{theorem}[spectral equivalence of multiplicative
  preconditioner]\label{thm:multi-precond} 
  Let $R_h$ be the Gauss-Seidel smoother for $A_{\lambda,h}$.
Under the Assumption \ref{prop:stab-decom} (stable decomposition),
the multiplicative preconditioner $P_{\rm{m}}$ defined in
\eqref{def:multi-precond} satisfies
\begin{equation*}\label{eq:multi-norm-equivalence} 
(P_{\rm{m}} A_{\lambda, h} v_{h}, v_{h} )_{\lambda, h} \simeq (v_{h}, v_{h})_{\lambda, h} \quad \forall v_{h} \in V_{h},
\end{equation*} 
with hidden constants independent of both $\lambda $ and $h$. That is,
$P_{\rm{m}}$ is a uniform spectral equivalence preconditioner of
$A_{\lambda, h}$.
\end{theorem}  
\begin{proof}
  \textit{Step (i):} For any $ v \in V_{h} $, denote $ w = (I -
  R_{h}'A_{\lambda, h})v $. By the definition of $ P_{\rm{m}} $
  \eqref{def:multi-precond}, we have 
   \begin{equation*} 
   \begin{aligned}
   (P_{\rm{m}}A_{\lambda, h}v, v)_{\lambda, h} &= (v, v)_{\lambda, h}
     - ((I - \Pi_{0}R_{0}\Pi_{0}^{\prime}A_{\lambda, h})w,w)_{\lambda,
     h} \\ 
   &= \|v\|_{\lambda, h}^{2} - \|w\|_{\lambda, h}^{2} +
     (\Pi_{0}R_{0}\Pi_{0}^{\prime}A_{\lambda, h}w, w)_{\lambda, h} \\ 
   &= \|v\|_{\lambda, h}^{2} - \|w\|_{\lambda, h}^{2} +
     (R_{0}A_{0}P_{0}w, P_{0}w)_{A_{0}}, 
   \end{aligned}
   \end{equation*} 
   where we use the identity \eqref{eq:A0P0} in the last step.
   Similarly, for $ \hat{P}_{\rm{m}} $ defined in
   \eqref{eq:multi-precond-hat}, we have 
   \begin{equation*} 
   (\hat{P}_{\rm{m}}A_{\lambda, h}v,v )_{\lambda, h} = \|v\|_{\lambda,
     h}^{2} - \|w\|_{\lambda, h}^{2} + (P_{0}w,P_{0}w)_{A_{0}}, 
   \end{equation*} 
   since $\hat{R}_0 = A_0^{-1}$. Invoking Lemma
   \ref{lm:spectral-equi-R0} (spectral equivalence of $R_{0}$), we
   have  
   \begin{equation}\label{eq:R0A0-equivalence} 
   C_{1}(P_{0}w , P_{0}w)_{A_{0}} \leq (R_{0}A_{0}P_{0}w,
     P_{0}w)_{A_{0}} \leq C_{2} (P_{0}w , P_{0}w)_{A_{0}}, 
   \end{equation} 
   where $ C_{1}, C_{2} $ are constants independent of $ \lambda $ and
   $ h $. Then,  
\begin{equation}\label{eq:multi-hat-equivalence} 
\min\{ 1, C_{1} \}(\hat{P}_{\rm{m}}A_{\lambda, h}v,v )_{\lambda, h}
  \leq (P_{\rm{m}}A_{\lambda, h}v,v)_{\lambda, h} \leq \max\{ 1,C_{2}
  \}(\hat{P}_{\rm{m}}A_{\lambda, h}v,v )_{\lambda, h}. 
\end{equation} 
Here, we use the contraction property of $R_{h}$, namely
$\|v\|_{\lambda, h}^{2} - \|w\|_{\lambda, h}^{2} \geq (1 - \|I - R_{h}'A_{\lambda, h}\|_{\lambda, h}^{2}) \|v\|_{\lambda,h}^{2} = (1 - \|I - R_{h}A_{\lambda, h}\|_{\lambda, h}^{2}) \|v\|_{\lambda,h}^{2} > 0 $.

\textit{Step (ii):} In light of \eqref{eq:multi-hat-equivalence}, we
  only need to show  
\begin{equation*}
    (\hat{P}_{\rm{m}}A_{\lambda, h}v,v )_{\lambda, h} \simeq
    (v,v)_{\lambda, h} \quad \forall v \in V_{h}, 
\end{equation*}
with hidden constants independent of $ \lambda $ and $ h $.  From the
identity \eqref{eq:A0P0}, we see that $(I - \Pi_{0}P_{0}) = (I -
\Pi_{0}P_{0})^{2}$ since $P_0 \Pi_0 = A_0^{-1}\Pi_0' A_{\lambda,h}
  \Pi_0 = I$. 
Then, we obtain the upper bound of $ \hat{P}_{\rm{m}}A_{\lambda, h} $:
\begin{equation*}\label{eq:upper-bound-hat} 
\begin{aligned}
  (\hat{P}_{\rm{m}}A_{\lambda, h}v,v)_{\lambda, h} &=
  \|v\|_{\lambda,h}^2 - ((I - \Pi_{0}P_{0})w,w)_{\lambda, h}\\
  &= \|v\|_{\lambda,h}^2 - ((I - \Pi_{0}P_{0})w,(I -
  \Pi_{0}P_{0})w)_{\lambda, h} \leq \|v\|_{\lambda,h}^2.
\end{aligned}
\end{equation*} 
Next, we estimate the lower bound of $ \hat{P}_{\rm{m}}A_{\lambda, h}
$ by Theorem \ref{thm:two-level} (two-level convergence rate) and
Assumption \ref{prop:stab-decom} (stable decomposition),  
\begin{equation*} 
\begin{aligned}
(\hat{P}_{\rm{m}}A_{\lambda, h}v,v)_{\lambda, h}  
  & = \|v\|_{\lambda,h}^2 - \|(I - \Pi_{0}P_{0})(I - R_{h}'A_{\lambda,
  h})v\|_{\lambda,h}^2 \\
  & \geq \Big(1 - \|(I - \Pi_{0}P_{0})(I - R_{h}'A_{\lambda,
  h})\|^{2}_{\lambda, h} \Big)\|v\|_{\lambda,h}^2 \\
  & =  \Big(1 - \|(I - R_{h}A_{\lambda,
  h})(I - \Pi_{0}P_{0})\|^{2}_{\lambda, h} \Big)\|v\|_{\lambda,h}^2
  ~~~~~~~ (\mbox{by }\eqref{eq:X-Z-dual})\\
  & =  \Big(1 - \|E\|^{2}_{\lambda, h} \Big)\|v\|_{\lambda,h}^2\\
  &= \displaystyle\frac{1}{K(V_{0})} \|v\|_{\lambda, h}^2 \geq
  c_{0}^{-2}\|v\|_{\lambda, h}^2.  
\end{aligned}
\end{equation*} 

Combining \textit{Step (i)} and \textit{Step (ii)}, we obtain
\begin{equation*}
    \min\{ 1, C_{1} \} c_{0}^{-2}\|v\|_{\lambda, h}^2 \leq
    (P_{\rm{m}}A_{\lambda, h}v,v)_{\lambda, h} \leq
    \max\{1,C_{2}\}\|v\|_{\lambda, h}^2.
\end{equation*}
The proof is thus complete.
\end{proof}

\begin{remark}
In the proof of Theorem \ref{thm:multi-precond} (spectral equivalence
 of multiplicative preconditioner), Assumption
\ref{prop:boundedness} (boundedness) is not directly used.  That is
because Lemma \ref{lm:spectral-equi-R0} (spectral equivalence of $
R_{0}$) leads to the boundedness on coarse space
  \eqref{eq:boundedness1}, and the boundedness on fine space
  \eqref{eq:boundedness2} is a direct consequence of the contraction
  property \eqref{eq:contraction-GS} of Gauss--Seidel smoother
  $R_{h}$.  
\end{remark}

\begin{remark}[computational complexity] \label{subsec:complexity}
We now discuss the computational complexity of the action of
preconditioner $P_{\rm{a}}$ \eqref{def:add-precond} and $P_{\rm{m}}$
\eqref{def:multi-precond}. Let $N_h$ be the number of degrees of
freedom, $N_p$ be the number of interior points of the grids. Since
the transfer operator $\Pi_0$ \eqref{def:Pi_h} is local, the action
except $R_0$ can be done within $\mathcal{O}(N_h)$ operations.

Invoking the definition of $R_{0}$ in \eqref{def:smoother-VL}, we can
see for any given $f_{h} \in V^{\prime}_{0}$ , $ u_{h} = R_{0}f_{h}$
can be obtained by solving two discrete Poisson-like equations
\eqref{eq:get-Rf}. The computational complexity of $(\lambda I -
\Delta_h)^{-1}$ with classic geometric multigrid methods was shown
to be optimal for $\lambda = \mathcal{O}(1)$
\cite{bramble1990parallel} and for arbitrary $\lambda > 0$
\cite{bramble2000computational}. For unstructured shape-regular grids
with $\lambda = \mathcal{O}(1)$, the computational complexity turns to
be $\mathcal{O}(N_p \log N_p)$ by constructing of an auxiliary coarse
grid hierarchy where the geometric multigrid can be applied
\cite{grasedyck2016nearly}. Therefore, the nearly optimal computational
complexity for arbitrary $\lambda>0$ on unstructured grids is to be
expected by combining the techniques from
\cite{bramble2000computational} and \cite{grasedyck2016nearly}.
\end{remark}

\begin{remark}[Implement of action $R_{0}$]
Let $ \{\psi_{i}\}_{i = 1}^{N_p}$ be the nodal basis functions of $
V_{0} $. Denote ${\bf A} = ( (\nabla \psi_{j},\nabla
  \psi_{i})_{\Omega})\in \mathbb{R}^{N_p \times N_p}$ , ${\bf M} = (
  (\psi_{j}, \psi_{i})_{\Omega}) \in \mathbb{R}^{N_p \times N_p}$ as
  the stiffness and mass matrix, respectively.
  For any $f_h \in V_{0}^{\prime}$, let ${\bf f} \in
  \mathbb{R}^{N_p} $ be its vector representation, i.e., $({\bf
  f})_{i} = \langle f_{h}, \psi_{i} \rangle $. Then $u_{h} =
  R_{0} f_{h}$ can be obtain by solving the following two linear
  systems successively. 
  \begin{subequations} \label{eq:Rf-matrix} 
  \begin{align*}
              {\bf f} & = (\lambda {\bf M}  + {\bf A} ) {\bf
              z},     \\
              {\bf M}{\bf z} & = (\lambda {\bf M} + {\bf A}) {\bf u},     
  \end{align*}
    \end{subequations}
    where ${\bf u} \in \mathbb{R}^{N_p} $ is the vector representation of
    $u_{h}$, i.e., $u_{h} = \sum_{i = 1}^{N_p}({\bf u})_{i} \psi_{i}$. 
\end{remark}


\section{Analysis of the auxiliary space
preconditioners}\label{sec:analysis}
In this section, we shall verify Assumption \ref{prop:stab-decom}
(stable decomposition) and Assumption \ref{prop:boundedness}
(boundedness), then show the proof of Lemma \ref{lm:spectral-equi-R0}
(spectral equivalence of $R_{0}$). For this propose, we first
show some properties about the space $ V_{0}$ introduced in
Section~\ref{sec:solver}. We refer to \cite{brenner2007mathematical, zhang2014optimal} for details on 
those results.   
\begin{lemma}[see \cite{zhang2014optimal}, Lemma 3.6 and \cite{brenner2007mathematical}]\label{lm:
  Clement}
    Let $ \Pi_{0}: V_{0} \to V_{h} $ be the interpolation operator
    defined in \eqref{def:Pi_h}. For any $ p_{h} \in V_{0} $ , it
    holds that 
\begin{equation}\label{eq:boundedness-Clement} 
\|\Pi_{0} p_{h}\|_{L^2(\Omega)} \lesssim \| p_{h} \|_{L^2(\Omega)}
  \quad\text{and}\quad |\Pi_{0} p_{h}|_{H^1(\Omega)} \lesssim | p_{h}
  |_{H^1(\Omega)}.
\end{equation} 
Moreover, we have
\begin{equation}\label{eq: err-Clement} 
\displaystyle\sum_{ T \in \mathcal{T}_{h}}^{} h_{T}^{-4}\|p_{h} -
  \Pi_{0} p_{h}\|_{L^2(T)}^{2} \lesssim \sum_{F \in
  \mathcal{F}_{h}^{i}}^{} h_{F}^{-1}
  \displaystyle\int_{F}^{}\llbracket \displaystyle\frac{\partial
  p_{h}}{\partial \boldsymbol{n}_{F}} \rrbracket^{2}\mathrm{d}s \qquad
  \forall p_{h} \in V_{0}. 
\end{equation} 
\end{lemma} 

\begin{lemma}\label{lm:discrete-laplace}
Suppose $\Omega \subset \mathbb{R}^{d}$ is a convex polytopal domain. Let $ \Delta_{h} $ be the
  discrete Laplacian operator defined in \eqref{eq:laplace-VL}.  Then
  it holds that 
    \begin{equation}\label{eq:discrete-laplace-equivalent} 
    \|\Delta_{h} p_{h}\|_{L^{2}(\Omega)}^{2} \simeq
      \displaystyle\sum_{F \in \mathcal{F}_{h}^{i}}^{} h_{F}^{-1}
      \displaystyle\int_{F}^{} \llbracket \displaystyle\frac{\partial
      p_{h}}{\partial \boldsymbol{n}_{F}} \rrbracket^{2}\mathrm{d}s
      \qquad \forall p_{h} \in V_{0}.
    \end{equation}  
\end{lemma} 
\begin{proof} 
  \eqref{eq:discrete-laplace-equivalent} can be obtained by combining a similar technique in  \cite[Lemma 3.1]{zhang2014optimal} and the elliptic regularity in convex polytopal domain in $\mathbb{R}^{d}$ \cite[Chapter 3]{grisvard2011elliptic}.    
\end{proof}
For any $ v_{h} \in V_{h} $, define nodal interpolation $ I_{h}: V_{h}
\to V_{0}$ as
\begin{equation}\label{def:nodal-interpolation} 
     I_{h}v_{h}(\boldsymbol{x}) = v_{h}(\boldsymbol{x})  \qquad
     \forall \boldsymbol{x} \in \mathcal{N}_{h}. 
\end{equation} 
According to standard polynomial approximation theory
\cite{brenner2007mathematical}, we have the following lemma. 
\begin{lemma}[see
  \cite{brenner2007mathematical}]\label{lm:nodal-interpolation}
    For any $ v_{h} \in V_{h} $, it holds that 
\begin{equation}\label{eq:booundness-nodal-interpolation} 
\|I_{h} v_{h}\|_{L^2(\Omega)} \lesssim \|v_{h}\|_{L^2(\Omega)}
  \quad\text{and}\quad |I_{h} v_{h}|_{H^1(\Omega)} \lesssim
  |v_{h}|_{H^1(\Omega)}.  
\end{equation} 
Moreover, for any $ T \in \mathcal{T}_{h} $ we have
  \begin{equation}\label{eq:err-nodal-interpolation}
  \|v_{h} - I_{h}v_{h}\|_{L^{2}(T)} \lesssim h^{2}_{T}
    |v_{h}|_{H^{2}(T)} \quad\text{and}\quad
    \|\displaystyle\frac{\partial  ( v_{h} - I_{h} v_{h})}{\partial
    \boldsymbol{n}} \|_{L^{2}(\partial T)} \lesssim h_{T}^{1/2}
    |v_{h}|_{H^{2}(T)}.
\end{equation} 
\end{lemma} 

\begin{lemma}\label{lm:Clement-Hermite}For any $ v \in V_{h} $, it holds that 
\begin{equation} 
    \displaystyle\sum_{T \in \mathcal{T}_{h}}^{}(h_{T}^{-2} + \lambda
    )^{2}\|v - \Pi_{0} I_{h} v \|_{L^{2}(T)}^{2} \lesssim
    \displaystyle\sum_{T \in \mathcal{T}_{h}}^{}(1 + \lambda h_{T}^{2}
    )^{2}\|D^{2} v\|_{L^{2}(T)}^{2}. 
\end{equation} 
\end{lemma} 
\begin{proof}
It suffics to show: 
    \begin{equation*}
        \displaystyle\sum_{T \in \mathcal{T}_{h}}^{}h_{T}^{-2\ell}
        \|v - \Pi_{0} I_{h} v \|_{L^{2}(T)}^{2} \lesssim
        \displaystyle\sum_{T \in \mathcal{T}_{h}}^{}h_{T}^{-2\ell +
        4}\|D^{2} v\|_{L^{2}(T)}^{2},
    \end{equation*} 
    for $ \ell = 0,1,2 $.  The case when $ \ell = 2 $ is proved in
    \cite[Lemma 3.7]{zhang2014optimal}. When $ \ell = 0,1 $, it can be
    proved by the same arguments in \cite[Lemma
    3.7]{zhang2014optimal}.  
\end{proof}

Next lemma gives a equivalence form of $ \|\cdot\|_{\bar{R}_{h}^{-1}} $
which will be used in the verification of Assumption
\ref{prop:stab-decom} (stable decomposition) and Assumption
\ref{prop:boundedness} (boundedness).
\begin{lemma}[norm equivalence of $ \bar{R}_{h}
  $]\label{lm:smoother-norm-equi}
  Let $ \mathcal{T}_{h} $ be a conforming shape regular triangulation
  of $ \Omega $.  Let $D_h^{-1}$ and $R_{h}$ be the Jacobi and Gauss-Seidel
  smoother for $A_{\lambda,h}$, respectively. Then, we have
    \begin{equation} \label{eq:norm-equ-Rh}
      \|v_{h}\|_{\bar{R}_{h}^{-1}}^{2} \simeq \|v_h\|_{D_h}^2
      \simeq \displaystyle\sum_{T
      \in \mathcal{T}_{h}}^{} ( h_{T}^{-2} + \lambda)^{2} \|
      v_{h}\|^{2}_{L^{2}(T)} \quad \forall v_{h} \in V_{h}. 
    \end{equation} 
    \end{lemma} 
    \begin{proof}
      By classical theory of iterative method \cite[Lemma
      4.6]{xu2017algebraic}, the symmetric Gauss-Seidel smoother and
      the Jacobi smoother are spectral equivalent for sparse SPD
      operator, namely $\|v_{h}\|_{\bar{R}_{h}^{-1}}^{2} \simeq
      \|v_h\|_{D_h}^2$.  Standard scaling argument
      \cite{brenner2007mathematical} gives that 
      \begin{subequations}
        \begin{align} 
          \|v_{h}\|^{2}_{L^{2}(T)} &\simeq
          \displaystyle\sum_{\omega_{\alpha } \supset T}^{} h_{T}^{2
            k_{\alpha} + 2} (\mathcal{N}_{\alpha}(v_{h}))^{2} \qquad
          \forall T \in \mathcal{T}_{h}, v_{h} \in V_{h}, \label{eq:norm-scaling} \\ 
    \displaystyle\sum_{T \in \mathcal{T}_{h}}
          \displaystyle\int_{T}^{}|D^{j} \varphi_{\alpha}|^{2}
          \mathrm{d}x &\simeq h_{\alpha}^{2 k_{\alpha} + 2 - 2j} \qquad
          j = 0,1,2,
          \label{eq:basis-scaling} 
        \end{align}
      \end{subequations}
    where $ \varphi_{\alpha} $ is the nodal basis function
      corresponding to the degree of freedom $
      \mathcal{N}_{\alpha}(\cdot)$. Combining
      \eqref{eq:norm-scaling} and \eqref{eq:basis-scaling}, we have
    \begin{equation}\label{eq:jacobi-smoother-D2u} 
    \begin{aligned}
    \displaystyle\sum_{T \in \mathcal{T}_{h}}^{} h_{T}^{-4}
      \|v_{h}\|_{L^{2}(T)}^{2}  &\simeq \displaystyle\sum_{T \in
      \mathcal{T}_{h}}^{} h_{T}^{-4} \displaystyle\sum_{\omega_{\alpha
      } \supset T}^{} h_{T}^{2 k_{\alpha} + 2}
      (\mathcal{N}_{\alpha}(v_{h}))^{2} \\
    & \simeq \displaystyle\sum_{\alpha}^{} (\# \omega_{\alpha}) h_{\alpha}^{2 k_{\alpha} - 2}(\mathcal{N}_{\alpha}(v_{h}))^{2} \\
    & \simeq  \displaystyle\sum_{\alpha}^{} \left( \displaystyle\sum_{T
      \in \mathcal{T}_{h}}^{} \displaystyle\int_{T}^{}|D^{2}
      \varphi_{\alpha}|^{2} \mathrm{d}x \right)
      (\mathcal{N}_{\alpha}(v_{h}))^{2}. 
    \end{aligned}
    \end{equation} 
    A similar argument leads to 
    \begin{equation}\label{eq:jacobi-smoother-Du} 
        \displaystyle\sum_{T \in \mathcal{T}_{h}}^{} h_{T}^{-2}
        \|v_{h}\|_{L^{2}(T)}^{2} \simeq \displaystyle\sum_{\alpha}^{}
        \left( \|\nabla \varphi_{\alpha} \|^{2}_{L^{2}(\Omega)}
        \right) (\mathcal{N}_{\alpha}(v_{h}))^{2},
    \end{equation}
    and 
    \begin{equation}\label{eq:jacobi-smoother-u} 
        \displaystyle\sum_{T \in \mathcal{T}_{h}}^{}
        \|v_{h}\|_{L^{2}(T)}^{2} \simeq \displaystyle\sum_{\alpha}^{}
        \left(  \| \varphi_{\alpha} \|^{2}_{L^{2}(\Omega)}  \right)
        (\mathcal{N}_{\alpha}(v_{h}))^{2}.
    \end{equation}
    Multiplying \eqref{eq:jacobi-smoother-D2u},
      \eqref{eq:jacobi-smoother-Du} and \eqref{eq:jacobi-smoother-u}
      respectively by $1$, $2\lambda$ and $\lambda^2$, then summing
      these equations, we obtain
$$ 
      \begin{aligned}
        & \quad ~\displaystyle\sum_{T \in \mathcal{T}_{h}}
        (h_T^{-2} + \lambda)^2\|v_{h}\|_{L^{2}(T)}^{2} \\
        &\simeq \displaystyle\sum_{\alpha} \left( \sum_{T\in
        \mathcal{T}_h} \int_T | D^2 \varphi_{\alpha}|^{2} + 2\lambda
        |\nabla \varphi_\alpha|^2 + \lambda^2 \varphi_\alpha^2 \mathrm{d}x
        \right) (\mathcal{N}_{\alpha}(v_{h}))^{2} = \|v_h\|_{D_h}^2,
      \end{aligned}
$$ 
which yields the norm equivalence \eqref{eq:norm-equ-Rh}.
\end{proof}
    
With the help of above lemmas, we are now ready to verify of
Assumption \ref{prop:stab-decom} (stable decomposition) and Assumption
\ref{prop:boundedness} (boundedness).
\begin{theorem}[verification of boundedness]\label{thm: boundedness-B} 
  There exist constants $c_{1}, c_{2}$ independent of $ \lambda $ and
  $h$, such that 
  \begin{subequations}
  \begin{align}
    \|v_{0}\|_{A_{0}} &\leq c_{1}
        \|v_{0}\|_{R^{-1}_{0}} \qquad \forall v_{0}\in V_{0}, 
\label{eq:boundedness-VL} \\ 
    \|v_{h}\|_{\lambda,h} &\leq c_{2}
    \|v_h\|_{\bar{R}_h^{-1}} \qquad \forall v_{h} \in V_{h}.
    \label{eq:boundedness-Vh} 
    \end{align}
  \end{subequations}
\end{theorem}
\begin{proof} 
\eqref{eq:boundedness-Vh} follows from the standard inverse estimate
  and Lemma \ref{lm:smoother-norm-equi} (norm equivalence of $ \bar{R}_{h}
  $), 
    \begin{equation*} 
    \begin{aligned}
    \|v_{h}\|_{\lambda, h}^{2} &= \displaystyle\sum_{T \in
      \mathcal{T}_{h}}^{}\|D^{2} v_{h}\|_{L^{2}(T)}^{2} +
      2\lambda\|\nabla
      v_{h}\|_{L^{2}(\Omega)}^{2}+\lambda^{2}\|v_{h}\|_{L^{2}(\Omega)}^{2}
      \\ 
    &\lesssim \displaystyle\sum_{T \in \mathcal{T}_{h}}^{} (h_{T}^{-4}
      + 2\lambda h_{T}^{-2} + \lambda^{2}) \|v_{h}\|_{L^{2}(T)}^{2} \\ 
    &= \displaystyle\sum_{T \in \mathcal{T}_{h}}^{}(h_{T}^{-2} +
      \lambda)^{2}\|v_{h}\|_{L^{2}(T)}^{2} \simeq 
      \|v_h\|_{\bar{R}_h^{-1}}^2.
    \end{aligned}
    \end{equation*} 
    Now turn to \eqref{eq:boundedness-VL}, by combining Lemma~\ref{lm: Clement}, Lemma~\ref{lm:discrete-laplace} and inverse estimate, we have 
    \begin{equation} \label{eq:boundedness-VL-1}
    \begin{aligned}
    \displaystyle\sum_{ T \in \mathcal{T}_{h}}^{}\|D^{2} \Pi_h v_{0}\|_{L^{2}(T)}^{2}& =\displaystyle\sum_{ T \in \mathcal{T}_{h}}^{}\|D^{2} (\Pi_h v_{0} - v_{0})\|_{L^{2}(T)}^{2} \\
    & \lesssim \displaystyle\sum_{T \in \mathcal{T}_{h}}^{} h_{T}^{-4}\|\Pi_h v_{0} - v_{0}\|_{L^{2}(T)}^{2} ~~~~(\mbox{by inverse estimate} )\\ 
    & \lesssim  \sum_{F \in \mathcal{F}_{h}^{i}}^{} h_{F}^{-1}\displaystyle\int_{F}^{} \llbracket \displaystyle\frac{\partial v_{0}}{\partial \boldsymbol{n}_{F}} \rrbracket^{2}\mathrm{d}s ~~~~~~~~(\mbox{by Lemma }\ref{lm: Clement} ) \\ 
    & \lesssim  \|\Delta_{h} v_{0}\|_{L^{2}(\Omega)}^{2}. ~~~~~~~~~~~~~~~~~~~~~(\mbox{by Lemma }\ref{lm:discrete-laplace}) 
    \end{aligned}
    \end{equation} 
    Then, \eqref{eq:boundedness-VL} follows from \eqref{eq:boundedness-VL-1} and boundedness of $ \Pi_{0} $ in Lemma~\ref{lm: Clement}
    \begin{equation} \label{boundedness-A0}  
    \begin{aligned}
        \|v_{0}\|^{2}_{A_{0}} & = \displaystyle\sum_{T \in \mathcal{T}_{h}}^{}\|D^{2} \Pi_{0} v_{0}\|_{L^{2}(T)}^{2} + 2\lambda\|\nabla \Pi_{0} v_{0}\|_{L^{2}(\Omega)}^{2}+\lambda^{2}\|\Pi_{0} v_{0}\|_{L^{2}(\Omega)}^{2} \\ 
        & \lesssim  \|\Delta_{h} v_{0}\|_{L^{2}(\Omega)}^{2}  ~~~~~~~~~~~~~~~~~~~~~~~~~~~~~~~ (\mbox{by } \eqref{eq:boundedness-VL-1}) \\
        &\quad + 2\lambda\|\nabla v_{0}\|_{L^{2}(\Omega)}^{2} + \lambda^{2}\|v_{0}\|_{L^{2}(\Omega)}^{2} ~~~~~~~(\mbox{by Lemma }\ref{lm: Clement}) \\ 
        & =  \|v_{0}\|^{2}_{R^{-1}_{0}}.
    \end{aligned}
    \end{equation} 
   This completes the proof. 
\end{proof}

\begin{theorem}[verification of stable decomposition] \label{thm:stab-decom}
  For any $ v \in V_{h} $, there exist $ v_{h} \in V_{h} $ and $ v_{0}
  \in V_{0} $ such that 
\begin{subequations}
  \begin{align}
  v &= v_{h} + \Pi_{0} v_{0}, \label{eq:stab-decom-1} \\ 
      \|v_{h}\|_{\bar{R}_{h}^{-1}}^{2} +
      \|v_{0}\|^{2}_{R_{0}^{-1}} &\leq  c^{2}_{0}\|v\|_{\lambda,
      h}^{2}, \label{eq:stab-decom-2}
  \end{align}
\end{subequations} 
where $ c_{0} $ is a constant independent with $ \lambda $ and $ h $.
\end{theorem} 

\begin{proof}
Recall the nodal interpolation $ I_{h}: V_{h} \to V_{0} $ defined in
\eqref{def:nodal-interpolation}. For any $v\in V_{h} $, take $ v_{0}
  = I_{h} v $ and $ v_{h} = v - \Pi_{0} I_{h} v $ so that
  \eqref{eq:stab-decom-1} is satisfied. It follows
    from Lemma~\ref{lm:Clement-Hermite}, Lemma~\ref{lm:smoother-norm-equi} (norm equivalence of $ \bar{R}_{h}
    $) and inverse estimate that 
    \begin{equation*}\label{eq:stab-decom-step1} 
    \begin{aligned}
        \|v_{h}\|^{2}_{\bar{R}^{-1}_{h}} & \simeq \displaystyle\sum_{T \in \mathcal{T}_{h}}^{}(h_{T}^{-2} + \lambda )^{2}\|v_{h}\|^{2}_{L^{2}(T)}\\ 
        &=  \displaystyle\sum_{T \in \mathcal{T}_{h}}^{}(h_{T}^{-2} + \lambda )^{2}\|v - \Pi_{0} I_{h} v \|_{L^{2}(T)}^{2} \\& \lesssim  \displaystyle\sum_{T \in \mathcal{T}_{h}}^{}(1 + \lambda h_{T}^{2} )^{2}\|D^{2} v\|_{L^{2}(T)}^{2}  ~~~~~~~~~~~~~~~~~~~~~~~~~~~~~~(\mbox{by Lemma }\ref{lm:Clement-Hermite}) \\
        & \lesssim  \displaystyle\sum_{T \in \mathcal{T}_{h}}^{}\|D^{2} v\|_{L^{2}(T)}^{2} + 2\lambda\|\nabla v\|^{2}_{L^{2}(\Omega)} + \lambda^{2} \|v\|^{2}_{L^{2}(\Omega)}~~~~~(\mbox{by inverse estimate}) \\
        & = \|v\|^{2}_{\lambda, h}.
    \end{aligned}
    \end{equation*} 

    On the other hand, combining Lemma~\ref{lm:discrete-laplace} and
    the approximation property of $ I_{h} $ in
    Lemma~\ref{lm:nodal-interpolation}, we have 
    \begin{equation}\label{eq:stab-decom-step2} 
    \begin{aligned}
    \|\Delta_{h} I_{h} v\|_{L^{2}(\Omega)}^{2} & \lesssim
      \displaystyle\sum_{F \in \mathcal{F}_{h}^{i}}^{} h_{F}^{-1}
      \displaystyle\int_{F}^{} \llbracket \displaystyle\frac{\partial
      I_{h} v}{\partial \boldsymbol{n}_{F}} \rrbracket^{2} \mathrm{d}s
      \\ 
    & \lesssim \displaystyle\sum_{F \in \mathcal{F}_{h}^{i}}^{}
      h_{F}^{-1} \displaystyle\int_{F}^{} \llbracket
      \displaystyle\frac{\partial (I_{h} v - v)}{\partial
      \boldsymbol{n}_{F}} \rrbracket^{2} \mathrm{d}s  +
      \displaystyle\sum_{F \in \mathcal{F}_{h}^{i}}^{} h_{F}^{-1}
      \displaystyle\int_{F}^{} \llbracket \displaystyle\frac{\partial
      v}{\partial \boldsymbol{n}_{F}} \rrbracket^{2} \mathrm{d}s \\ 
    & \lesssim \displaystyle\sum_{T \in \mathcal{T}_{h}}^{} \|D^{2}
      v\|_{L^{2}(T)}^{2}.
    \end{aligned}
    \end{equation} 
Here, in the last step, the standard scaling argument
  \cite{brenner2007mathematical} gives that 
\begin{equation*} 
     h_{F}^{-1} \displaystyle\int_{F}^{} \llbracket
     \displaystyle\frac{\partial v}{\partial \boldsymbol{n}_{F}}
     \rrbracket^{2} \mathrm{d}s \lesssim \displaystyle\sum_{T \in
     \{T^{+}, T^{-}\}}^{} \|D^{2} v\|^{2}_{L^{2}(T)}, 
\end{equation*}   
holds for any interior face $ F = \partial T^{+} \cap \partial T^{-}
  $, where the $C^0$-continuity at face and $C^1$-continuity at
  $(d-2)$-dimensional subsimplex guarantee that the piecewise linear
  function on $\omega_F  = T^+ \cup T^-$ has to be a linear function
  on the $\omega_F$.  

Combining the boundedness of $ I_{h}$ in
  Lemma~\ref{lm:discrete-laplace} and \eqref{eq:stab-decom-step2}, we
  get 
    \begin{equation*}\label{eq:stab-decom-step3} 
    \begin{aligned}
    \|v_{0}\|_{R_{0}^{-1}}^{2} 
    & = \|\Delta_{h} I_{h} v\|^{2}_{L^{2}(\Omega)} + 2\lambda\|\nabla
      I_{h} v\|^{2}_{L^{2}(\Omega)} + \lambda^{2}\|I_{h}
      v\|^{2}_{L^{2}(\Omega)}  \\ 
    & \lesssim  \displaystyle\sum_{T \in \mathcal{T}_{h}}^{}  \|D^{2}
      v\|^{2}_{L^{2}(T)} ~~~~~~~~~~~~~~~~~~~~(\mbox{by }
      \eqref{eq:stab-decom-step2}) \\
      & \quad + 2\lambda\|\nabla  v\|^{2}_{L^{2}(\Omega)} +
      \lambda^{2}\|v\|^{2}_{L^{2}(\Omega)} ~~~~(\mbox{by
      } \eqref{eq:booundness-nodal-interpolation}) \\ 
    & = \|v\|_{\lambda, h}^{2}.
    \end{aligned}
    \end{equation*} 
The proof is thus complete.
\end{proof}

By the similar arguments in Theorem \ref{thm: boundedness-B}
(verification of boundedness) and Theorem \ref{thm:stab-decom}
(verification of stable decomposition), we are now ready to
give the proof of Lemma \ref{lm:spectral-equi-R0} (spectral
equivalence of $ R_{0} $).

\begin{proof}[Proof of Lemma \ref{lm:spectral-equi-R0} (spectral
  equivalence of $R_{0}$)]
Note that $v_{0} = I_{h}\Pi_{h}v_{0}$ for any $v_0 \in V_0$, then
\begin{equation*} 
\begin{aligned}
\|v_{0}\|^{2}_{R_{0}^{-1}}  
&= \|\Delta_{h} I_{h} \Pi_{h}v_{0}\|^{2}_{L^{2}(\Omega)} +
2\lambda\|\nabla I_{h} \Pi_{h}v_{0}\|^{2}_{L^{2}(\Omega)} +
\lambda^{2}\|I_{h} \Pi_{h}v_{0}\|^{2}_{L^{2}(\Omega)}  \\ 
&\lesssim \displaystyle\sum_{T \in \mathcal{T}_{h}}^{}  \|D^{2}
  \Pi_{h}v_{0}\|^{2}_{L^{2}(T)}
  ~~~~~~~~~~~~~~~~~~~~~~~~~~~~(\mbox{similar to }
  \eqref{eq:stab-decom-step2}) \\ 
& \quad + 2\lambda\|\nabla \Pi_{h}v_{0}\|^{2}_{L^{2}(\Omega)} +
  \lambda^{2}\|\Pi_{h}v_{0}\|^{2}_{L^{2}(\Omega)}
  ~~~~~~~~(\mbox{by } \eqref{eq:booundness-nodal-interpolation})\\ 
        &= \|\Pi_{h}v_{0}\|^{2}_{\lambda, h} = \|v_{0}\|_{A_{0}}^{2} . 
    \end{aligned}
    \end{equation*} 
The other direction has been proved in \eqref{boundedness-A0}, whence
we obtain the spectral equivalence $\|v_{0}\|_{A_{0}} \simeq
\|v_{0}\|_{R_{0}^{-1}}$.    
\end{proof}


\section{Numerical experiments} \label{sec:numerical}
In this section, we present numerical experiments to illustrate the
performance of PFMRES preconditioners for solving both linear and
nonlinear problems.

Denote $\kappa(\cdot)$ for the condition number of operator or matrix,
and DOF for the number of degrees of freedom. On the fine level, we
use Gauss-Seidel method for $ A_{\lambda, h} $ with three iterations
as the smoother on $V_{h}$, which shares the similar properties as the
Gauss-Seideal smoother $R_h$. For the actor of subspace smoother
$R_{0}$, we apply algebraic multigrid method (AMG) with stop criterion
$\|{\bf b} - {\bf A}{\bf x}\|_2 / \|{\bf b}\|_2 \leq 10^{-8}$ for a
linear sytem ${\bf A}{\bf x} = {\bf b}$. 

\subsection{Preconditioning effect of \texorpdfstring{$ P_{\rm{a}}
$}{Pa} and \texorpdfstring{$ P_{\rm{m}} $}{Pm} for \texorpdfstring{$
A_{\lambda, h} $}{A}} \label{subsec:ex1}

We test the theoretical results in Theorem \ref{thm:add-precond}
(spectral equivalence  of additive preconditioner) and Theorem
\ref{thm:multi-precond} (spectral equivalence of multiplicative
preconditioner) by examining the condition number of $
P_{\rm{m}}A_{\lambda, h} $ and $ P_{\rm{a}}A_{\lambda, h}$. To
showcase the flexibility of the preconditioner on non-uniform girds,
we illustrate the performance of preconditioners on a sequence of
graded bisection grids $ \{ \mathcal{T}_{\ell} \}_{\ell \in
\mathbb{N}_{0}} $ with grading factor $ 1/2 $ on $ \Omega = (-1,1)^{2}
$, see Fig. \ref{fig:test1-graded-mesh}. More specifically, we mark
the elements which satisfy $ |T| > C (\|x_{T}\|_{\ell_{2}} - 1/2)^{2}
/ \# \mathcal{T}_{\ell}$, where $ x_{T} $ is barycenter of the element
$ T $, $ \# \mathcal{T}_{\ell} $ is the number of elements in $
\mathcal{T}_{\ell} $. In the experiment, we set $ C = 1000 $.
Further, in the case of additive preconditioners, we apply the scaled
form \eqref{eq:omega-additive} in Remark \ref{rmk:omega-additive}
(additive preconditioner with scaled parameter) with $\omega = 1/10$. 
 
\begin{figure}[!htbp]
  \centering 
  \captionsetup{justification=centering}
  \subfloat[Level 3: DOF = 2,387]{
    \includegraphics[width=0.30\textwidth]{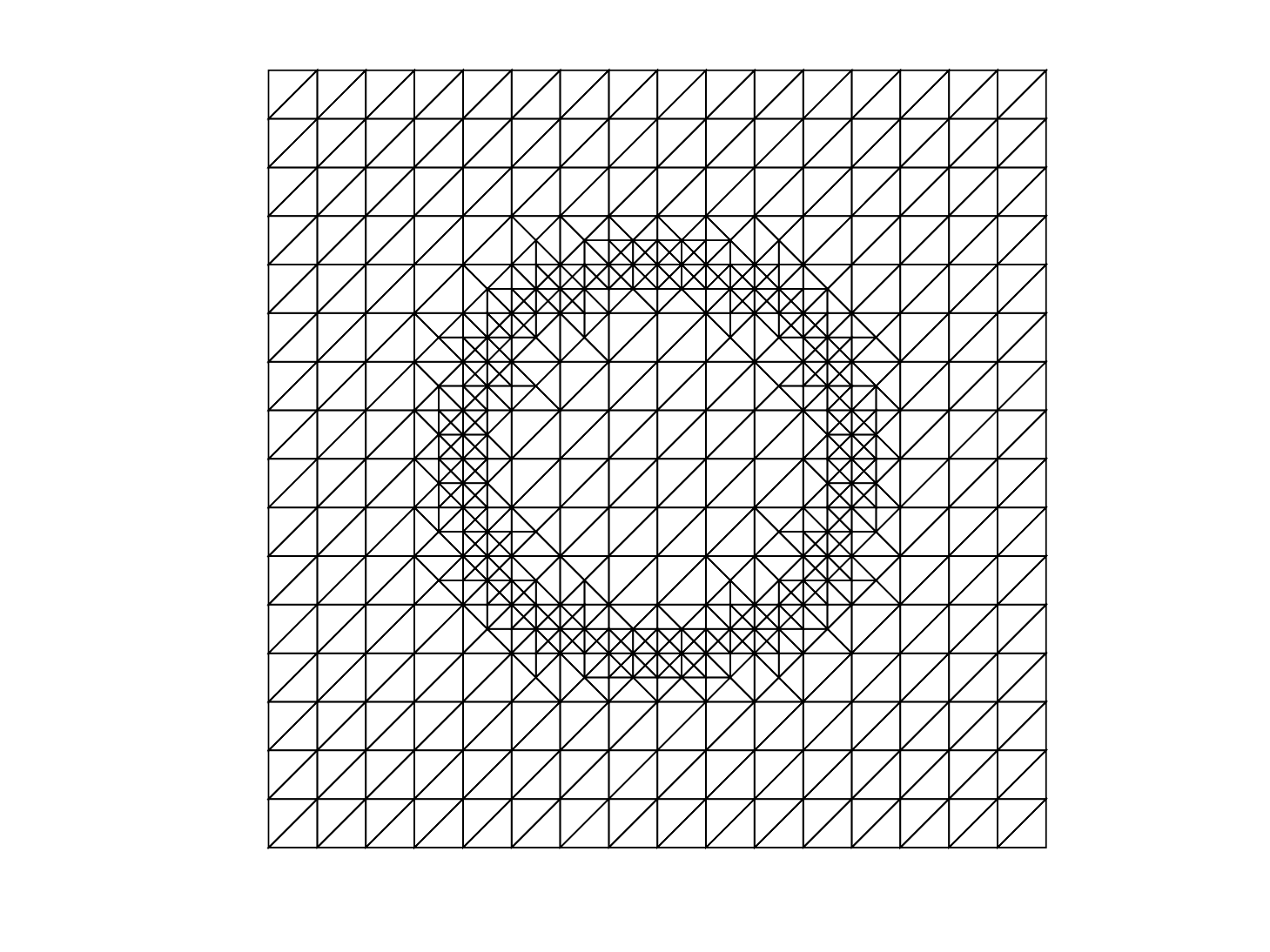}
    \label{fig:test1-level3}
  } %
  \subfloat[Level 5: DOF = 4,467]{
    \includegraphics[width=0.30\textwidth]{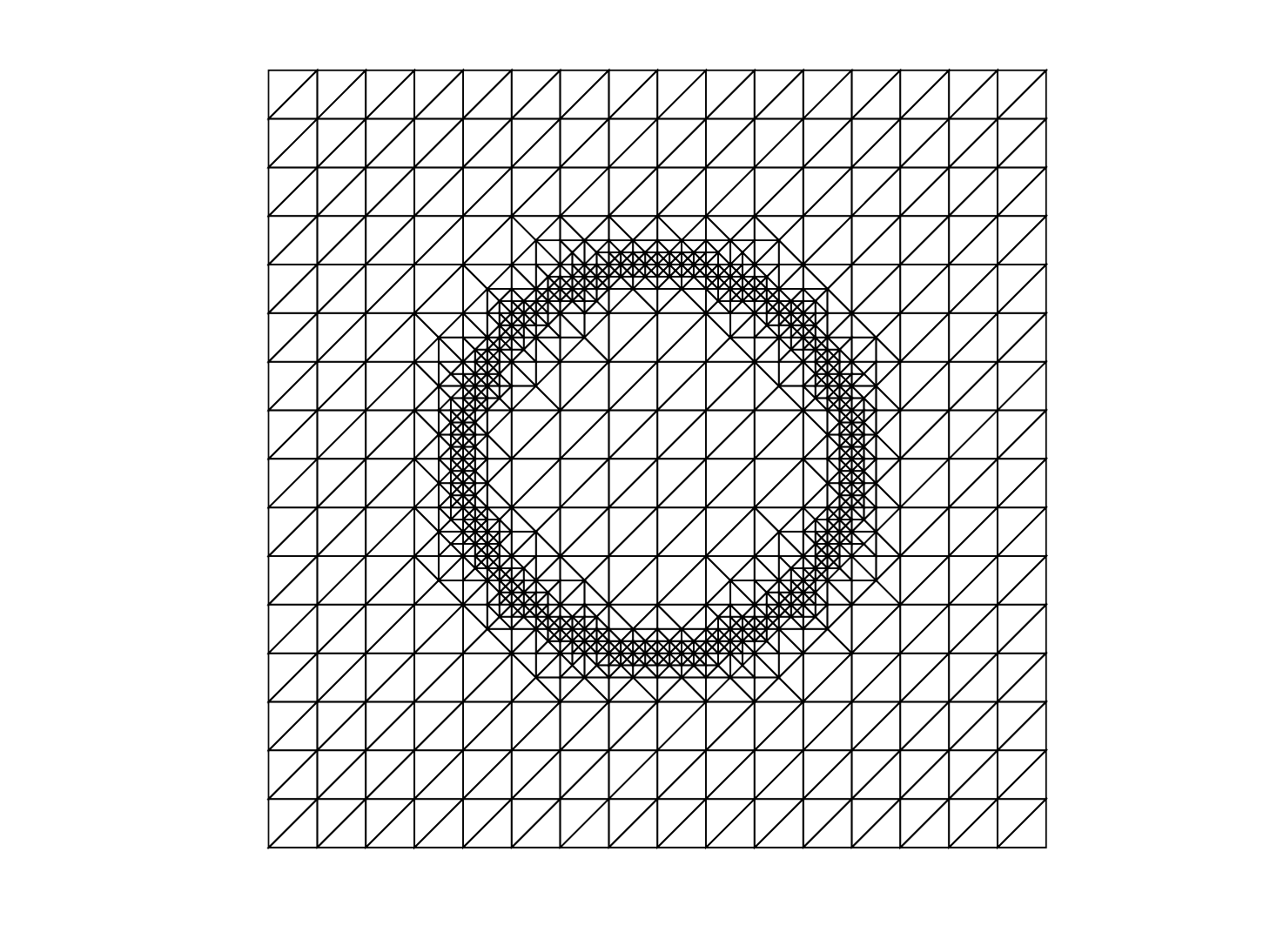}
    \label{fig:test1-level5}
  } %
  \subfloat[Level 7: DOF = 10,027]{
    \includegraphics[width=0.30\textwidth]{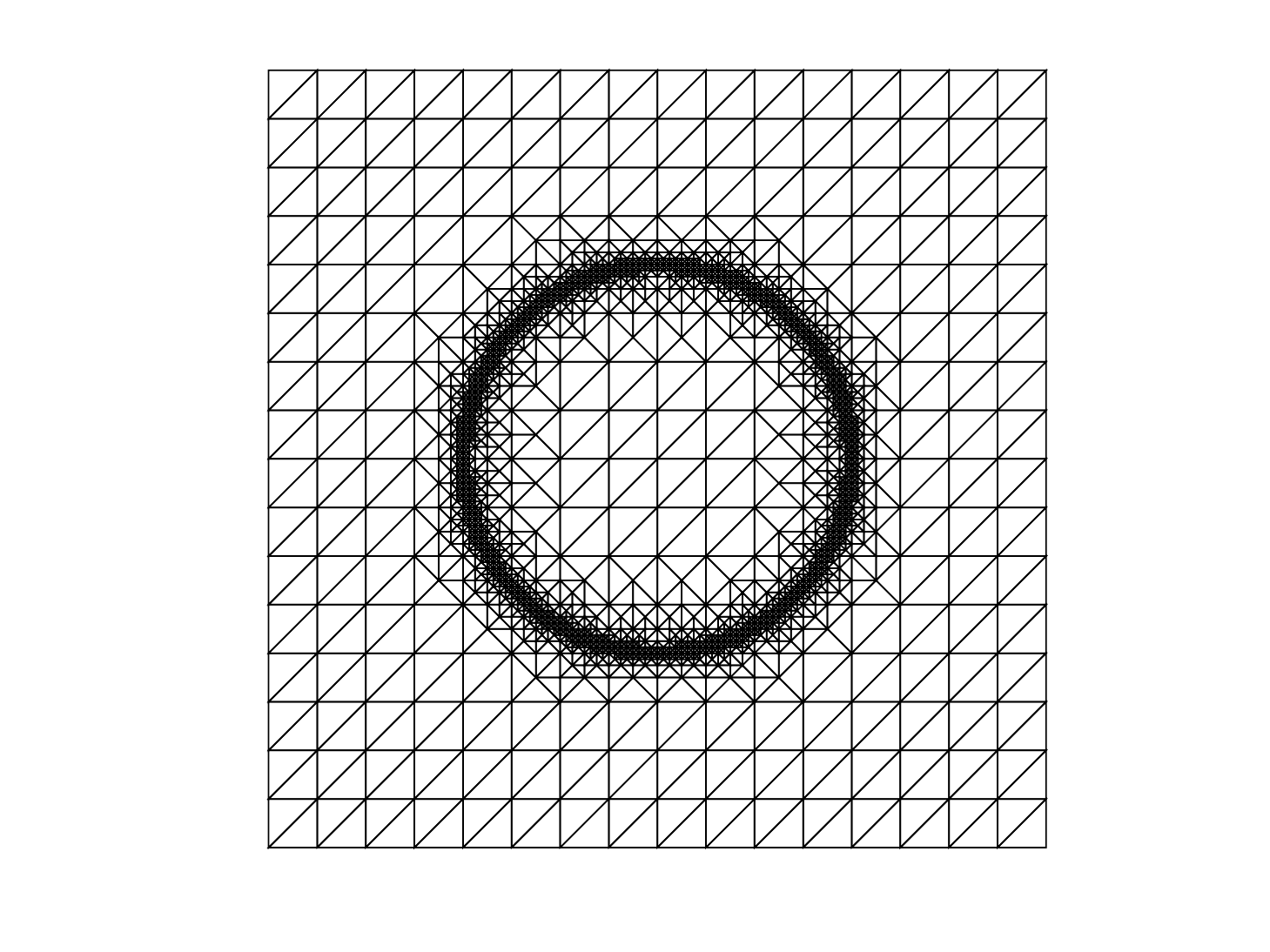}
    \label{fig:test1-level7}
  } %
  \caption{Graded grids at different levels for Experiment
  \ref{subsec:ex1}, initial grid size $h_0 = 1/8$.}
  \label{fig:test1-graded-mesh}
\end{figure}
  
The resulting condition numbers for additive and multiplicative
preconditioners at different bisection levels are listed respectively
in Table \ref{tab:test1-add} and Table
\ref{tab:test1-multi}.  We observe that both $ P_{\rm{m}} $ and $
P_{\rm{a}} $ are uniform preconditioners with respect to both
$\lambda$ and DOF, which is in agreement with the theoretical results
in Theorem \ref{thm:multi-precond} (spectral equivalence of
multiplicative preconditioner) and Theorem
\ref{thm:add-precond} (spectral equivalence of additive preconditioner). 
We also observe that the multiplicative
preconditioners perform better than additive ones. 

\begin{table}[!htbp]
    \centering
    \caption{Condition number of additive preconditioning for Experiment
    \ref{subsec:ex1}.}
  \begin{tabular}{|c|c|c|c|c|c|c|c|}
      \hline
  \multirow{2}[4]{*}{DOF} & \multicolumn{7}{|c|}{$
  \kappa(P_{\rm{a}}A_{\lambda, h}) $} \\
  \cline{2-8}
     &$\lambda=10^{-3}$ & $\lambda=10^{-2}$ & $\lambda=10^{-1}$
     &$\lambda=1$ &$\lambda=10$ & $\lambda=10^{2}$ & $\lambda=10^{3}$
     \\ \hline
      3,147 	& 1.64e2 	& 1.64e2 	& 1.60e2 	& 1.40e2 	& 9.41e1 	& 5.43e1
      & 2.48e1 \\ \hline  	 
      4,467 	& 1.65e2 	& 1.64e2 	& 1.60e2 & 1.40e2 	& 9.49e1 	& 6.43e1
      & 3.42e1 	\\ \hline
      6,587 &1.61e2 & 1.60e2 & 1.57e2 & 1.35e2 & 9.22e1&
      7.09e1 & 4.52e1 \\ \hline	
      10,027 &  1.61e2 	& 1.61e2 	& 1.57e2 & 1.36e2 	& 9.62e1 	&
      8.04e1 	& 5.49e1	 \\ \hline 
      15,927 	& 1.62e2 	& 1.62e2 	& 1.58e2 	& 1.36e2 	& 9.95e1 	&
      8.78e1  	& 6.46e1  	\\ \hline 
    \end{tabular}
    \label{tab:test1-add}
\end{table}

\begin{table}[!htbp]
\centering
\caption{Condition number of multiplicative preconditioning for Experiment
  \ref{subsec:ex1}.}
\begin{tabular}{|c|c|c|c|c|c|c|c|}
    \hline
    \multirow{2}[4]{*}{DOF} & \multicolumn{7}{|c|}{$
    \kappa(P_{\rm{m}}A_{\lambda, h}) $} \\
\cline{2-8}
&$\lambda=10^{-3}$ & $\lambda=10^{-2}$ & $\lambda=10^{-1}$
  &$\lambda=1$ &$\lambda=10$ & $\lambda=10^{2}$ & $\lambda=10^{3}$  \\ 
    \hline
    3,147 	& 5.76  	& 5.76  	& 5.76  	& 5.75  	& 5.65  	& 5.21
    & 4.16  	\\ \hline 
    4,467 	& 5.76 	& 5.76 	& 5.75& 5.75 	& 5.68 	& 5.42  	& 4.77  	\\
    \hline
    6,587 	& 5.63  	& 5.63  	& 5.63  	& 5.62  	& 5.56  	& 5.39
    & 4.99 	\\ \hline 
    10,027 	& 6.03  	& 6.03  	& 6.03  	& 6.03  	& 6.02  	& 5.94
    & 5.59   \\\hline
    15,927 	& 6.07  	& 6.07  	& 6.07  	& 6.07  	& 6.06  	& 6.00
    & 5.75  \\ \hline
\end{tabular}
\label{tab:test1-multi}
\end{table}

\subsection{Uniform preconditioning for the linearised problems}
\label{subsec:ex2}

In the second experiment, we consider the linearised problems in the
semi-smooth Newton steps, i.e.,  elliptic equations in non-divergence
form: 
\begin{equation}\label{eq:nondiv} 
A:D^{2}u + \boldsymbol{b}^{\theta}\cdot\nabla u - c^{\theta} u =
  f^{\theta}, 
\end{equation} 
on the domain $\Omega = (-1,1)^{2}$. The coefficients are set to be 
\begin{equation}\label{eq:test1-coef} 
A = \begin{pmatrix}
 2 & \displaystyle\frac{x_{1}x_{2}}{|x_{1} x_{2}|} \\
 \displaystyle\frac{x_{1}x_{2}}{|x_{1} x_{2}|} & 2 \\ 
\end{pmatrix} 
\quad \boldsymbol{b}^{\theta} = \sqrt{\theta}(x_{1}, x_{2})^{T} \quad
  c^{\theta} = 3\theta, 
\end{equation}
where $\theta$ is a parameter so that the $\lambda$ in Cordes condition differs
with varying $\theta$.  Let exact solution be
\begin{equation}\label{eq:test1-solu} 
  u(x_{1}, x_{2}) = (x_{1}\mathrm{e}^{1 - |x_{1}|} - x_{1})(x_{2}
  \mathrm{e}^{1- |x_{2}|} - x_{2}),  
\end{equation} 
the right hand side $ f^{\theta} $ is directly calculated from the
equation \eqref{eq:nondiv}. For any given $\theta > 0$, we may  set $
\lambda = \theta $, which yields 
\begin{equation}\label{eq:test1-cordes} 
\displaystyle\frac{|A|^{2} + |\boldsymbol{b}^{\theta}|^{2}/2\lambda +
  (c^{\theta}/\lambda)^{2}}{(\operatorname{Tr}A + c^{\theta}/\lambda
  )^{2}} = \displaystyle\frac{10 + 1/2(x_{1}^{2} + x^{2}_{2}) + 9}{(4
  +3)^{2}} \leq \displaystyle\frac{20}{49},
\end{equation} 
which means that the Cordes condition is satisfied for $ \varepsilon =
9/20 $. The PGMRES method is applied to solve the linear system
arising from the discretization of \eqref{eq:nondiv} with varying
parameter $ \theta > 0 $. We stop the iteration when the relative
residual is smaller than $10^{-6}$. The iteration numbers for
PGMRES are shown in Tables \ref{tab:test-2-add} and
\ref{tab:test-2-multi}. Similarly, a better performance of
multiplicative preconditioner is observed.  

\begin{table}[!htbp]
  \centering
  \caption{The iteration steps of additive preconditioning for Experiment \ref{subsec:ex2}.}
\resizebox{\textwidth}{!}{ %
\begin{tabular}{|c|c|c|c|c|c|c|c|}
    \hline
     \multirow{2}[4]{*}{DOF} &$\lambda=10^{-3}$ & $\lambda=10^{-2}$ & $\lambda=10^{-1}$ & $\lambda=1$ &$\lambda=10^{1}$ & $\lambda=10^{2}$ & $\lambda=10^{3}$ \\
   \cline{2-8}
   & steps & steps  & steps  & steps & steps  & steps  & steps  \\ 
    \hline
    5,055 & 120 	& 119 	& 116 		& 117 	& 99 	& 71 	& 36 	\\  \hline
    20,351 	& 133 	& 132 	& 132 	& 130 	& 114 	& 95 	& 59 	\\  \hline	
    81,663 	& 137 	& 137 	& 136 & 134 	& 118 	& 109 	& 72 	\\  \hline
    32,7167 	& 138 	& 137 	& 135& 135 	& 118 	& 117 	& 80    \\ 
     \hline
    \end{tabular}%
    }
  \label{tab:test-2-add}%
 \end{table}%
\begin{table}[!htbp]
    \centering
    \caption{The iteration steps of multiplicative preconditioning for Experiment \ref{subsec:ex2}. }
  \resizebox{\textwidth}{!}{ %
  \begin{tabular}{|c|c|c|c|c|c|c|c|}
      \hline
      \multirow{2}[4]{*}{DOF}  &$\lambda=10^{-3}$ & $\lambda=10^{-2}$ & $\lambda=10^{-1}$ & $\lambda=1$ &$\lambda=10^{1}$ & $\lambda=10^{2}$ & $\lambda=10^{3}$ \\
     \cline{2-8}
     & steps & steps  & steps  & steps & steps  & steps  & steps  \\ 
      \hline 
      5,055 	& 28 	& 28 	& 28 & 28 	& 28 	& 27 	& 24 	\\  \hline
      20,351 & 26 	& 26 	& 26	& 26 	& 25 	& 25 	& 29 	\\  \hline
      81,663 & 25 	& 24 	& 24 	& 24 	& 23 	& 23 	& 28 	\\  \hline 	
      32,7167 	& 23 	& 23 	& 23 & 22 	& 21 	& 21 	& 26 	\\ 
       \hline
      \end{tabular}%
      }
    \label{tab:test-2-multi}%
   \end{table}%

\subsection{Application to the HJB equations}
   In this experiment, we solve the nonlinear HJB equations
   \eqref{eq:HJB} on the domain $ \Omega = (0,1)^{2}$. Following
   \cite{smears2014discontinuous}, we take $ \Lambda =
   [0, \pi/3] \times \operatorname{SO} (2) $, where $
   \operatorname{SO}(2) $ is the set of $ 2 \time 2 $ rotation
   matrices. The coefficients are given by $ \boldsymbol{b}^{\alpha} =
   0 $, $ c^{\alpha} = \pi^{2} $, and 
   \begin{equation}\label{eq:test2-coef} 
   A^{\alpha} = \displaystyle\frac{1}{2} \sigma^{\alpha} (\sigma^{\alpha})^{T}, \qquad \sigma^{\alpha} = R^{T}\begin{pmatrix}
    1 &   \sin \theta\\
     0  &  \cos \theta \\
   \end{pmatrix} , \qquad \alpha = (\theta, R) \in \Lambda. 
   \end{equation}
   We choose $ f^{\alpha} = \sqrt{3} \sin^{2} \theta / \pi^{2} + g $,
   $g$ independent of $ \alpha $ such that the exact solution of the HJB
   equations \eqref{eq:HJB} is $ u(x_{1}, x_{2}) =
   \operatorname{exp}(x_{1}x_{2}) \sin(\pi x_{1} ) \sin (\pi x_{2})
   $. That is
   \begin{equation}\label{eq:test2-g} 
   g = \sup_{\alpha \in \Lambda}\{ A^{\alpha}: D^{2} u - c^{\alpha} u - \sqrt{3}\sin^{2} \theta /\pi^{2} \}. 
   \end{equation}    

   In the semi-smooth Newton algorithm, the initial guess is $u_h^0 =
   0$.  For the discrete linearised systems arising from each Newton
   step, the multiplicative preconditioners is applied. We compute the
   average number of PGMERS iterations of per Newton step which
   require to reduce the residual norm below a relative tolerance of
   {$ 10^{-4} $.  Convergence of the Newton method was determined by
   requiring a step-increment $ L^{2} $-norm below $ 10^{-6} $. These
   tolerances are chosen to balance the different sources of error
   originating from discretization.

   \begin{table}[htbp]
    \centering
    \caption{Average PGMRES iterations (Newton steps)} 
  \begin{tabular}{c c c}
     \hline
      $\rm{DOF}$ & $ h $ & Average PGMRES iterations (Newton steps)\\
     \hline
     	 71 & 	 1/4 & 14 (6)  	\\ 
       303 & 	 1/8 & 18 (6)    \\
       1,247 & 	1/16 & 18 (6)  \\
       5,055 & 	 1/32 & 18 (7) \\
       20,351 & 	 1/64 & 18 (8) 	 \\
        \hline 
      \end{tabular}%
    \label{tab:test-3-avgGmres}%
   \end{table}%

   The numbers of semi-smooth Newton iterations and average PGMRES
   iterations are listed in Table \ref{tab:test-3-avgGmres}.  As can
   be observed from \cite{smears2014discontinuous,wu2021c0}, the
   semi-smooth Newton algorithm convergences fast (within eight
   iterations in the numerical experiment). In each Newton step, we
   apply the PGMRES with multiplicative preconditioner due to its
   better performance than the additive one.  Based on the results
   shown in the Table \ref{tab:test-3-avgGmres}, we can conclude that
   our multiplicative preconditioner is also effective and robust in the
   application to the HJB equations.


\section*{Acknowledgments}
The authors would like to express their gratitude to Prof. Jun Hu in Peking University for his helpful discussions. 

\bibliographystyle{siamplain}
\bibliography{HermiteAux} 

\begin{thebibliography}{10}

\bibitem{bramble2000computational}
{\sc J.~Bramble, J.~Pasciak, and P.~Vassilevski}, {\em Computational scales of
  sobolev norms with application to preconditioning}, Mathematics of
  Computation, 69 (2000), pp.~463--480.

\bibitem{bramble1990parallel}
{\sc J.~Bramble, J.~Pasciak, and J.~Xu}, {\em Parallel multilevel
  preconditioners}, Mathematics of Computation, 55 (1990), pp.~1--22.

\bibitem{brenner1999convergence}
{\sc S.~Brenner}, {\em Convergence of nonconforming multigrid methods without
  full elliptic regularity}, Mathematics of computation, 68 (1999), pp.~25--53.

\bibitem{brenner2007mathematical}
{\sc S.~Brenner and R.~Scott}, {\em The {M}athematical {T}heory of {F}inite
  {E}lement {M}ethods}, vol.~15, Springer Science \& Business Media, 2007.

\bibitem{brenner1989optimal}
{\sc S.~C. Brenner}, {\em An optimal-order nonconforming multigrid method for
  the biharmonic equation}, SIAM journal on numerical analysis, 26 (1989),
  pp.~1124--1138.

\bibitem{brenner1996two}
{\sc S.~C. Brenner}, {\em A two-level additive {S}chwarz preconditioner for
  nonconforming plate elements}, Numerische Mathematik, 72 (1996),
  pp.~419--447.

\bibitem{brenner2020adaptive}
{\sc S.~C. Brenner and E.~L. Kawecki}, {\em {Adaptive $C^0$ interior penalty
  methods for Hamilton--Jacobi--Bellman equations with Cordes coefficients}},
  Journal of Computational and Applied Mathematics,  (2020), p.~113241.

\bibitem{carstensen2021hierarchical}
{\sc C.~Carstensen and J.~Hu}, {\em Hierarchical argyris finite element method
  for adaptive and multigrid algorithms}, Computational Methods in Applied
  Mathematics,  (2021).

\bibitem{christiansen2018nodal}
{\sc S.~H. Christiansen, J.~Hu, and K.~Hu}, {\em Nodal finite element de {R}ham
  complexes}, Numerische Mathematik, 139 (2018), pp.~411--446.

\bibitem{elman1982iterative}
{\sc H.~C. Elman}, {\em Iterative methods for large, sparse, nonsymmetric
  systems of linear equations}, PhD thesis, Yale University New Haven, Conn,
  1982.

\bibitem{fleming2006controlled}
{\sc W.~H. Fleming and H.~M. Soner}, {\em Controlled Markov processes and
  viscosity solutions}, vol.~25, Springer Science \& Business Media, 2006.

\bibitem{gallistl2019mixed}
{\sc D.~Gallistl and E.~S\"{u}li}, {\em Mixed finite element approximation of
  the {H}amilton-{J}acobi-{B}ellman equation with {C}ordes coefficients}, SIAM
  Journal on Numerical Analysis, 57 (2019), pp.~592--614.

\bibitem{glowinski1979numerical}
{\sc R.~Glowinski and O.~Pironneau}, {\em Numerical methods for the first
  biharmonic equation and for the two-dimensional {S}tokes problem}, SIAM
  review, 21 (1979), pp.~167--212.

\bibitem{grasedyck2016nearly}
{\sc L.~Grasedyck, L.~Wang, and J.~Xu}, {\em A nearly optimal multigrid method
  for general unstructured grids}, Numerische Mathematik, 134 (2016),
  pp.~637--666.

\bibitem{grisvard2011elliptic}
{\sc P.~Grisvard}, {\em {E}lliptic {P}roblems in {N}onsmooth {D}omains}, SIAM,
  2011.

\bibitem{holst1997schwarz}
{\sc M.~Holst and S.~Vandewalle}, {\em Schwarz methods: to symmetrize or not to
  symmetrize}, SIAM Journal on Numerical Analysis, 34 (1997), pp.~699--722.

\bibitem{kawecki2019dgfem}
{\sc E.~L. Kawecki}, {\em {A DGFEM for nondivergence form elliptic equations
  with Cordes coefficients on curved domains}}, Numerical Methods for Partial
  Differential Equations, 35 (2019), pp.~1717--1744.

\bibitem{kawecki2020unified}
{\sc E.~L. Kawecki and I.~Smears}, {\em {Unified analysis of discontinuous
  Galerkin and $C^0$-interior penalty finite element methods for
  Hamilton--Jacobi--Bellman and Isaacs equations}}, arXiv preprint
  arXiv:2006.07202,  (2020).

\bibitem{kawecki2020convergence}
{\sc E.~L. Kawecki and I.~Smears}, {\em {Convergence of adaptive discontinuous
  Galerkin and $C^0$-interior penalty finite element methods for
  Hamilton--Jacobi--Bellman and Isaacs equations}}, Foundations of
  Computational Mathematics,  (2021).

\bibitem{maugeri2000elliptic}
{\sc A.~Maugeri, D.~K. Palagachev, and L.~G. Softova}, {\em Elliptic and
  {P}arabolic {E}quations with {D}iscontinuous {C}oefficients}, vol.~109,
  WILEY-VCH Verlag GmbH \& Co., 2000.

\bibitem{neilan2015discrete}
{\sc M.~Neilan}, {\em Discrete and conforming smooth de {R}ham complexes in
  three dimensions}, Mathematics of Computation, 84 (2015), pp.~2059--2081.

\bibitem{neilan2019discrete}
{\sc M.~Neilan and M.~Wu}, {\em Discrete {M}iranda-{T}alenti estimates and
  applications to linear and nonlinear {PDE}s}, Journal of Computational and
  Applied Mathematics, 356 (2019), pp.~358--376.

\bibitem{notay2013further}
{\sc Y.~Notay and A.~Napov}, {\em Further comparison of additive and
  multiplicative coarse grid correction}, Applied Numerical Mathematics, 65
  (2013), pp.~53--62.

\bibitem{peisker1988numerical}
{\sc P.~Peisker}, {\em On the numerical solution of the first biharmonic
  equation}, ESAIM: Mathematical Modelling and Numerical Analysis, 22 (1988),
  pp.~655--676.

\bibitem{peisker1987conjugate}
{\sc P.~Peisker and D.~Braess}, {\em A conjugate gradient method and a
  multigrid algorithm for {M}orley's finite element approximation of the
  biharmonic equation}, Numerische Mathematik, 50 (1987), pp.~567--586.

\bibitem{peisker1990iterative}
{\sc P.~Peisker, W.~Rust, and E.~Stein}, {\em Iterative solution methods for
  plate bending problems: Multigrid and preconditioned {CG} algorithm}, SIAM
  journal on numerical analysis, 27 (1990), pp.~1450--1465.

\bibitem{renardy2006introduction}
{\sc M.~Renardy and R.~C. Rogers}, {\em An {I}ntroduction to {P}artial
  {D}ifferential {E}quations}, vol.~13, Springer Science \& Business Media,
  2006.

\bibitem{saad2003iterative}
{\sc Y.~Saad}, {\em Iterative {M}ethods for {S}parse {L}inear {S}ystems}, SIAM,
  2003.

\bibitem{smears2018nonoverlapping}
{\sc I.~Smears}, {\em Nonoverlapping domain decomposition preconditioners for
  discontinuous {G}alerkin approximations of {Hamilton--Jacobi--Bellman}
  equations}, Journal of Scientific Computing, 74 (2018), pp.~145--174.

\bibitem{smears2013discontinuous}
{\sc I.~Smears and E.~S\"{u}li}, {\em Discontinuous {G}alerkin finite element
  approximation of nondivergence form elliptic equations with {C}ordes
  coefficients}, SIAM Journal on Numerical Analysis, 51 (2013), pp.~2088--2106.

\bibitem{smears2014discontinuous}
{\sc I.~Smears and E.~S\"{u}li}, {\em Discontinuous {G}alerkin finite element
  approximation of {H}amilton-{J}acobi-{B}ellman equations with {C}ordes
  coefficients}, SIAM Journal on Numerical Analysis, 52 (2014), pp.~993--1016.

\bibitem{smears2016discontinuous}
{\sc I.~Smears and E.~S{\"u}li}, {\em Discontinuous {G}alerkin finite element
  methods for time-dependent {H}amilton-{J}acobi-{B}ellman equations with
  {C}ordes coefficients}, Numerische Mathematik, 133 (2016), pp.~141--176.

\bibitem{stevenson2003analysis}
{\sc R.~Stevenson}, {\em An analysis of nonconforming multigrid methods,
  leading to an improved method for the {Morley} element}, Mathematics of
  computation, 72 (2003), pp.~55--81.

\bibitem{wu2021c0}
{\sc S.~Wu}, {\em {$C^0$} finite element approximations of linear elliptic
  equations in non-divergence form and {Hamilton--Jacobi--Bellman equations
  with Cordes coefficients}}, Calcolo, 58 (2021).

\bibitem{xu1996auxiliary}
{\sc J.~Xu}, {\em The auxiliary space method and optimal multigrid
  preconditioning techniques for unstructured grids}, Computing, 56 (1996),
  pp.~215--235.

\bibitem{xu2017algebraic}
{\sc J.~Xu and L.~Zikatanov}, {\em Algebraic multigrid methods}, Acta Numerica,
  26 (2017), pp.~591--721.

\bibitem{zhang2014optimal}
{\sc S.~Zhang and J.~Xu}, {\em Optimal solvers for fourth-order {PDE}s
  discretized on unstructured grids}, SIAM Journal on Numerical Analysis, 52
  (2014), pp.~282--307.

\bibitem{zhang1994multilevel}
{\sc X.~Zhang}, {\em Multilevel {S}chwarz methods for the biharmonic
  {D}irichlet problem}, SIAM Journal on Scientific Computing, 15 (1994),
  pp.~621--644.

\end{thebibliography}

\end{document}